\documentclass[12pt,twoside,a4paper]{article}
\usepackage[english]{babel}
\usepackage[utf8]{inputenc}
\usepackage{amsmath}
\usepackage{amsfonts}
\usepackage{amsthm}
\usepackage{amssymb}
\usepackage{bm}
\usepackage{tikz}
\usepackage{xr}
\usepackage{chngcntr}
\usepackage{enumitem}
\usepackage{authblk}
\usepackage[title]{appendix}
\usepackage{chngcntr}
\usepackage{apptools}

\newcommand{\subscript}[2]{$#1 _ #2$}

\newtheorem*{rmk}{Remark}
\newtheorem*{rmks}{Remarks}
\newtheorem{prop}{Proposition}

\newtheorem*{exams}{Examples}
\newtheorem{theorem}{Theorem}
\newtheorem{lemma}{Lemma}
\newtheorem*{defn}{Definition}
\newtheorem{cor}{Corollary}
\newtheorem*{thm*}{Theorem}
\newtheorem{conj}{Conjecture}
\usepackage{theoremref}

\DeclareMathOperator{\homm}{Hom}

\DeclareMathOperator{\diag}{diag}
\DeclareMathOperator{\rank}{rank}

\newcommand{\Sum}[2]{\displaystyle\sum_{#1}^{#2}}

\newcommand{\Z}{\mathbb{Z}}
\newcommand{\N}{\mathbb{N}}
\newcommand{\C}{\mathbb{C}}
\newcommand{\Q}{\mathbb{Q}}

\usepackage{tocloft}

\numberwithin{theorem}{section}
\numberwithin{prop}{section}
\numberwithin{cor}{section}
\numberwithin{lemma}{section}

\renewcommand{\thetheorem}{\arabic{section}.\arabic{theorem}}

\renewcommand{\thecor}{\arabic{section}.\arabic{cor}}

\title{Ax-Schanuel for $GL_n$}
\author{Georgios Papas}
\date{}
\AtEndDocument{%
	\par
	\medskip
	\begin{tabular}{@{}l@{}}%
		\textsc{Dept. of Mathematics, University of Toronto, Toronto, Canada.}\\
		\textit{E-mail address}: \texttt{george.papas@mail.utoronto.ca}
\end{tabular}}

\begin{document}
\maketitle

\begin{abstract} In this paper we prove an Ax-Schanuel type result for the exponential functions of general linear groups over $\mathbb{C}$. We prove the result first for the group of upper triangular matrices and then for the group $GL_n$ of all $n\times n$ invertible matrices over $\mathbb{C}$. We also obtain Ax-Lindemann type results for these maps as a corollary, characterizing the bi-algebraic subsets of these maps.
\end{abstract}

\section{Introduction}

We study questions related to the functional transcedence of exponential functions of $n\times n$ matrices over $\C$. To be more precise, motivated by the exposition in \cite{pilafun}, we prove Ax-Schanuel and Ax-Lindemann type Theorems for the exponential function of upper triangular matrices, as well as the exponential function of general $n\times n$ matrices over $\C$. We have divided our exposition into two parts dealing with each of the cases separately.

In the most general case we will consider the exponential function $E:\mathfrak{gl}_n\rightarrow GL_n$ over $\C$. The strongest result that we achieve in this case is the following
\begin{thm*}[Two-sorted Weak Ax-Schanuel for $GL_n$] Let $U\subset \mathfrak{gl}_n$ be a bi-algebraic subvariety that contains the origin, let $X=E(U)$, and let $V\subset U$ and $Z\subset X$ be algebraic subvarieties, such that $\vec{0}\in V$ and $I_n\in Z$. If $C$ is a component of $V\cap E^{-1}(Z)$ with $\vec{0}\in C$, then, assuming that $C$ is not contained in any proper weakly special subvariety of $U$,
	\begin{center}$\dim_\C C\leq \dim_\C V+\dim_\C Z-\dim_\C X$.
\end{center}\end{thm*}

Here the term \textbf{component} of a subset $R\subset \mathfrak{gl}_n$ refers to a complex-analytically irreducible component of $R$, while the term \textbf{bi-algebraic} refers to a subvariety $U$ of $\mathfrak{gl}_n$ whose image $E(U)$ under the exponential is such that its Zariski closure $Zcl(E(U))$ satisfies $\dim_{\C}(E(U))=\dim_{\C}(Zcl(E(U)))$. The \textbf{ weakly-special} subvarities will be defined later on and are characterized, as we will see, by an Ax-Lindemann-type statement. In that sense, they are naturally defined for the exponential map of matrices, mirroring the definition of weakly special subvarities for other transcendental maps, for more on those we refer to \cite{pilaAO}.

\begin{rmks} 1. The restrictions on $V$, $C$, and $Z$, requiring that $\vec{0}\in C$, $\vec{0}\in V$, and $I_n\in Z$ are needed to deal with the existence of positive dimensional connected components in the preimage of $I_n$. A phenomenon that does not appear in other transcendental maps considered so far in the literature, at least to the knowledge of the author. For example, in $\mathfrak{h}_2$ all matrices of the form $\begin{pmatrix}
		2k\pi i && x\\
		0 && 2l\pi i
	\end{pmatrix}$ with $k\neq l$ integers are mapped to the identity matrix $I_2$ via the exponential. We return to this in \ref{applications}.\\

2. It is worth noting that, in contrast to \cite{kirby2009theory} and \cite{ax1971schanuel}, our target space, the group $GL_n$, is no longer a commutative group and that the exponential map is no longer a group homomorphism. We are nevertheless able to extract ``functional equations" satisfied by our map, those will be reflected in the weakly special subvarieties.
\end{rmks}

\subsection*{A short review of Ax-Schanuel and Ax-Lindemann}

Our main motivation is the Ax-Schanuel Theorem for the usual exponential function of complex numbers. This result, originally a conjecture of Schanuel, is due to J. Ax \cite{ax1971schanuel}. Ax's proof employs techniques of differential algebra. One of the equivalent formulations of this theorem is the following

\begin{thm*}[Ax-Schanuel] Let $y_1,\ldots,y_n\in\C[[t_1,\ldots,t_m]]$ have no constant terms. If the $y_i$ are $\Q-$linearly independent then 
	
	\begin{center}
		$tr.d._\C \C(y_1,\ldots,y_n,e^{y_1},\ldots,e^{y_n})\geq n+\rank\big(\frac{d y_{\nu}}{d t_\mu}\big)_{\underset{\nu=1,\ldots,n}{\scriptscriptstyle{\mu=1,\ldots,m}}}.$
	\end{center}
\end{thm*}

An immediate consequence of the above Ax-Schanuel Theorem is the characterization of all the bi-algebraic subsets of $\C^n$ with respect to the map $\pi:\C^n\rightarrow (\C^\times)^n$, given by $(z_1\ldots,z_n)\mapsto (e^{z_1},\ldots,e^{z_n})$. In other words it leads to a characterization of the subvarieties $W\subset \C^n$ with the property that $dim_\C(\pi(W))=dim_\C(Zcl(\pi(W)))$.

\begin{defn}A subvariety $W$ of $\C^n$ will be called weakly special, or geodesic, if it is defined by any number $l\in\N$  of equations of the form\begin{center}
		$\Sum{i=i}{n}q_{i,j}z_i=c_j$, $j=1,\ldots,l$,
	\end{center}where $q_{i,j}\in\Q$ and $c_j\in\C$.
\end{defn}

This characterization of bi-algebraic sets is due to the following result, dubbed Ax-Lindemann by Pila due to its resemblance to Lindemann's theorem,
\begin{thm*}[Ax-Lindemann]Let $V\subset (\C^\times)^n$ be an algebraic subvariety. Then any maximal algebraic subvariety $W\subset \pi^{-1}(V)$ is weakly special.
\end{thm*}
For more on these notions, along with a proof of Ax-Lindemann as a corollary of Ax-Schanuel, we refer to \cite{pilafun}.

Subsequent results in functional transcendence that look to achieve similar results to the above theorem for other transcendental functions have also been dubbed as ``Ax-Schanuel" and ``Ax-Lindemann" respectively. Ax-Schanuel results are known for affine abelian group varieties, due to J.Ax \cite{MR0435088}, for semiabelian varieties, due to J. Kirby \cite{kirby2009theory}, the $j-$function, due to J. Pila and J. Tsimerman \cite{pila2016ax}, for Shimura varieties, due to N. Mok, J. Pila, and J. Tsimerman \cite{mok2017ax}, for variations of Hodge structures, due to B. Bakker and J. Tsimerman \cite{bakker2017ax}, and for mixed Shimura varieties due to Z. Gao \cite{gao2018mixed}. Finally, B. Klingler, E. Ullmo, and A. Yafaev \cite{klingler2013hyperbolic} have proven an Ax-Lindemann result for any Shimura variety.

\subsection*{Summary of results in Part I}

The first part of this paper deals with the exponential of the algebra $\mathfrak{h}_n$ of $n\times n$ upper triangular matrices. This map is more accessible to computations. These computations form the technical part of the reduction from the Ax-Schanuel result in this case to the original Ax-Schanuel Theorem and are presented in \ref{tritosection}.

We let $E:\mathfrak{h}_n\rightarrow U_n$ denote the exponential of $\mathfrak{h}_n$, $U_n$ being the group of upper triangular invertible matrices over $\C$. Let $A$ be an upper triangular matrix with entries in $\C[[t_1,\ldots,t_m]]$. We will denote the field extension of $\C$ that results from adjoining to $\C$ the entries of both matrices $A$ and $E(A)$ by $\C(A,E(A))$. In this case our main result will be 

\begin{thm*}[Weak Ax-Schanuel for $U_n$] Let $f_1,\ldots,f_n,g_{i,j}\in \mathbb{C}[[t_1,\ldots,t_m]]$ be power series, where $1\leq i<j\leq n$. We assume that the $f_i$ do not have a constant term. Let $A$ be the $n\times n$ upper triangular matrix with diagonal $\vec{f}$ and the $(i,j)$ entry equal to $g_{i,j}$. Let $N=\dim_\Q\langle f_1,\ldots,f_n\rangle_{\Q}$, then   \begin{center}
		$tr.d._\mathbb{C}\mathbb{C}(A,E(A))\geq N+\rank(J(\vec{f},\vec{g};\vec{t}))$.\end{center}
\end{thm*}

Here $\langle f_1,\ldots,f_n\rangle_{\Q}$ denotes the linear span of the $f_i$ over $\Q$, while $J(\vec{f},\vec{g};\vec{t})$ denotes the $\frac{n(n+1)}{2}\times m$ Jacobian matrix with entries of the form $\partial h_s/\partial t_j$, where $h_s$, with $1\leq s\leq \frac{n(n+1)}{2}$, is an ordering of the $g_{i,j}$ and the $f_i$. The rank of the Jacobian is its rank over the fraction field of $\mathbb{C}[[t_1,\ldots,t_m]]$.

The main idea is that given a matrix $A\in\mathfrak{h}_n$ we are able to canonically choose a basis of generalized eigenvectors for it, based solely on the multiplicities of its eigenvalues. This basis is chosen in such a way that makes it computable in terms of the entries of the matrix $A$. At the same time we can determine the action of the matrix $A$ in each of its generalized eigenspaces.

The canonical basis and its properties lead us naturally to define the notion of eigencoordinates in \ref{tritosection}. These will roughly be coordinates describing the generalized eigenspaces of a matrix $A$ along with the action of the matrix $A$ in each of these spaces. Ultimately they will be used in reducing the Ax-Schanuel result to the original Ax-Schanuel Theorem. 

\subsubsection*{Weakly special subvarieties for $\mathfrak{h}_n$}

After establishing the Ax-Schanuel result we turn towards characterizing the bi-algebraic subvarieties of the exponential map $E:\mathfrak{h}_n\rightarrow U_n$ that contain the origin. In the literature for other transcendental maps such subvarieties are referred to as \textbf{weakly special}.

To that end we start by defining the weakly special subvarieties of $\mathfrak{h}_n$ that contain the origin in \ref{wesp}. Roughly speaking a subvariety $V\subset \mathfrak{h}_n$ that contains the origin will be weakly special if its diagonal coordinates satisfy $\Q-$linear relations, while the other algebraic relations on it come from algebraic relations on the eigencoordinates, i.e. from algebraic relations between generalized eigenvectors and the actions of matrices in $V$ on their generalized eigenspaces. 

These expectations are based on two properties of the exponential of a matrix. First, that if $v$ is an $f-$generalized eigenvector for the matrix $A$ then $v$ is also an $e^f-$generalized eigenvector for its exponential, the matrix $E(A)$. Secondly, the exponential is a bi-algebraic map between nilpotent and unipotent operators, the inverse being the logarithm. So the nilpotent action defined by $A$ on a generalized eigenspace gets mapped bi-algebraically to the corresponding action of $E(A)$ on the same space.

As a corollary of our Ax-Schanuel result we obtain an Ax-Lindemann-type result. This result will imply that the weakly special sets we define will be exactly the bi-algebraic subsets of $\mathfrak{h}_n$ that contain the origin. We present two examples of weakly special subvarieties as a motivation for the proper definition later on in \ref{wesp}.
\begin{exams}1. Let $V_1\subset \mathfrak{h}_3$ be the set of all upper triangular $3\times 3$ matrices $A$ satisfying the following conditions:\begin{enumerate}[label=(\subscript{C}{{\arabic*}})]
		
		\item $A$ has diagonal $(f_1,f_2,f_3)$ with $f_i\neq f_j$ for all $i\neq j$,
		\item $\vec{v}_2=(1,1,0)$ is an $f_2-$eigenvector, and
		\item there exists $s\in\C$ such that $\vec{v}_3=(s,s^2,1)$ is an $f_3-$eigenvector.
	\end{enumerate}	
Notice that all upper triangular matrices will have $\vec{v}_1=(1,0,0)$ as an eigenvector for the eigenvalue $f_1$. We also note that here the choice of ``$s^2$" is arbitrary and can be replaced by any algebraic function of $s$.
	
If we set $W_1=Zcl(V_1)$ to be the Zariski closure of $V_1$ in $\mathfrak{h}_n$, then $W_1$ will be a weakly special subvariety of $\mathfrak{h}_n$. 

To see that this is natural to expect, consider the set $X_1\subset U_3$ of all invertible $3\times3$ upper triangular matrices $A$ satisfying:\begin{enumerate}[label=(\subscript{D}{{\arabic*}})]
	
	\item $A$ has diagonal $(f_1,f_2,f_3)$ with $f_i\neq f_j$ for all $i\neq j$,
	\item $\vec{v}_2=(1,1,0)$ is an $f_2-$eigenvector, and
	\item there exists $s\in\C$ such that $\vec{v}_3=(s,s^2,1)$ is an $f_3-$eigenvector.
\end{enumerate} 

This set will be contained in $E(V_1)$, by the remarks above. In fact, $X_1$ is a Zariski open subset of $E(V_1)$ and $dim_\C(X_1)=dim_\C(E(V_1))=dim_\C(Zcl(E(V_1)))=4$.\\

2. Once again let us denote by $\vec{v}_1$ the vector $(1,0,0)$. We consider $V_2\subset\mathfrak{h}_3$ to be the set of all upper triangular matrices $A$ satisfying the following conditions:
\begin{enumerate}[label=(\subscript{C}{{\arabic*}})]
\item  $A$ has diagonal $(f_1,f_2,f_1)$ wtih $f_1\neq f_2$,
\item there exists $s\in \C$ such that $\vec{v}_2=(s,1,0)$ is an $f_2-$eigenvector,
\item for the same $s$ as above, the vector $\vec{v}_3=(0,s^2,1)$ is a generalized eigenvector for the eigenvalue $f_1$, and 
\item for the same $s$ as above, we have that \begin{center}
	$A\vec{v}_3=f_1\vec{v}_3+s^5\vec{v}_1$.
\end{center}\end{enumerate}
Here again the choices of ``$s^2$" and ``$s^5$" in $C_3$ and $C_4$ are arbitrary and can be replaced by any algebraic function in $s$. 

Setting $W_2=Zcl(V_2)$ we get a weakly special subvariety of $\mathfrak{h}_3$.

We argue as before, by considering the set $X_2\subset U_3$ of all invertible $3\times3$ upper triangular matrices $A$ satisfying:\begin{enumerate}[label=(\subscript{D}{{\arabic*}})]
	\item  $A$ has diagonal $(f_1,f_2,f_1)$ wtih $f_1\neq f_2$,
	\item there exists $s\in \C$ such that $\vec{v}_2=(s,1,0)$ is an $f_2-$eigenvector,
	\item for the same $s$ as above, the vector $\vec{v}_3=(0,s^2,1)$ is a generalized eigenvector for the eigenvalue $f_1$, and 
	\item for the same $s$ as above, we have that \begin{center}
		$A\vec{v}_3=f_1\vec{v}_3+f_1 s^5\vec{v}_1$.
\end{center}\end{enumerate}

The change in $D_4$ from $s^5$ to $f_1s^5$ is related to the action of the exponential on the nilpotent operator on the $f_1-$generalized eigenspace. It turns out that $X_2$ is a Zariski open subset of $E(V_2)$ and $dim_\C(X_2)=dim_\C(E(V_2))=dim_\C(Zcl(E(V_2)))=3$.\\

\end{exams}

\subsection*{Summary of results in Part II}

In the second part we deal with the same questions of functional transcendence this time for the exponential map $E:\mathfrak{gl}_n\rightarrow GL_n$ of the algebra of $n\times n$ matrices over $\C$.

In this case we generalize the picture we had in $\mathfrak{h}_n$. Instead of a specific canonical basis and eigencoordinates, we introduce the notion of the \textbf{data} of a matrix $A\in\mathfrak{gl}_n$. This new notion will effectively have the role that the eigencoordinates had for $\mathfrak{h}_n$.

The data of a matrix $A$ will comprise of the distinct eigenvalues of $A$, their multiplicities, their generalized eigenspaces, and the nilpotent operators defined by the matrix $A$ on each such generalized eigenspace. Given the number $k$ of distinct eigenvalues and the multiplicity $m_i$, $i=1,\ldots,k$, of each of them, the rest of the above information, i.e. the generalized eigenspaces and corresponding nilpotent operators, will be parametrized by an affine variety, which we will denote by $W_{k}(\vec{m})$. With the help of $W_k(\vec{m})$, we shall see that the Ax-Schanuel result for $E$ is reduced to the original Ax-Schanuel Theorem. 

Let $A$ be an $n\times n$ matrix with entries in $\C[[t_1,\ldots,t_m]]$. As before we will denote the field extension of $\C$ that results from adjoining to $\C$ the entries of the matrices $A$ and $E(A)$ by $\C(A,E(A))$. The result we obtain will then be 

\begin{thm*}[Weak Ax-Schanuel for $GL_n$]\thlabel{asgln} Let $g_{i,j}\in\C[[t_1,\ldots,t_m]]$ be power series with no constant term, where $1\leq i,j\leq n$. Let $f_i$, where $1\leq i\leq n$, denote the eigenvalues of the matrix $A=(g_{i,j})$. Let us also set $N=\dim_\Q\langle f_1,\ldots,f_n\rangle_{\Q}$, then 
	\begin{center}$tr.d._\C \C(A,E(A))\geq N+\rank J((g_{i,j});\vec{t})$.\end{center}\end{thm*}

\subsubsection*{Weakly special subvarieties for $\mathfrak{gl}_n$}

Again the above Ax-Schanuel result leads us to a description of the weakly special subvarieties of $\mathfrak{gl}_n$ that contain the origin. These are defined in detail in \ref{wekspegln}. Roughly speaking these will be subvarieties of $\mathfrak{gl}_n$ that are subject to algebraic relations of the following two types:\begin{enumerate}
	\item $\Q-$linear relations on the eigenvalues and
	\item algebraic relations on the coordinates of a variety $W_k(\vec{m})$, as above, for some $k\in \N$ and $\vec{m}\in\N^k$, or in other words, relations coming from a subvariety $W\subset W_k(\vec{m})$.
\end{enumerate}

In other words, the relations are either on the eigenvalues, or between the generalized eigenspaces and the corresponding nilpotent operators defined on them. All the while there can be no algebraic relations between eigenvalues and generalized eigenspaces or eigenvalues and nilpotent operators defined on those spaces. Alternatively, we require that the two types of relations considered above do not interfere with one another. 

These results for the Lie algebra $\mathfrak{gl}_n$ will imply, as a corollary, Ax-Schanuel and Ax-Lindemann type results for all subalgebras $\mathfrak{g}$ of $\mathfrak{gl}_n$ and their respective exponentials.

Finally, we note that Ax-Schanuel and Ax-Lindemann type statements are important tools in proving statements of unlikely intersections, such as Andr\'e-Oort and Zilber-Pink type statements following the Pila-Zannier method. With that in mind, in the final section, we conclude with some suggestions for further questions with a view towards applications of our results to questions in unlikely intersections in non-commutative linear algebraic groups.\\

\textbf{Acknowledgements:} I would like to thank my advisor Jacob Tsimerman for introducing me to the subject, for many helpful discussions, and for reading through earlier versions of this paper and making helpful suggestions and pointing to some errors. I would also like to thank Edward Bierstone and Abhishek Oswal for helpful remarks and useful discussions regarding the results of this paper. Finally, I would like to thank the anonymous referee for many helpful suggestions and remarks, and especially the proof of \thref{asmainred}.

\section{Ax-Schanuel in families}

In each of the two cases we deal with, we start by approaching the Ax-Schanuel result from a functional standpoint. We then use properties of the exponential maps in question to reduce to the following corollary of the classic Ax-Schanuel Theorem

\begin{prop}[Weak Ax-Schanuel in families]\thlabel{asmainred}Let $f_i,g_{j}\in \mathbb{C}[[t_1,\ldots,t_m]]$, where $1\leq i\leq n$, $1\leq j\leq k$, be power series. Then, assuming that the $f_i$ are $\mathbb{Q}-$linearly independent modulo $\mathbb{C}$,\begin{center}
		$tr.d._\mathbb{C}\mathbb{C}(\{f_i,g_{j},e^{f_i}:1\leq i\leq n, 1\leq j\leq k\})\geq n+\rank(J(\vec{f},\vec{g};\vec{t}))$.\end{center}\end{prop}
\begin{proof}\footnote{The author thanks the anonymous referee for this proof that helped shorten this section significantly.}Without loss of generality, we may assume that the set \begin{center}
		$\{f_i,g_j:1\leq i\leq n, 1\leq j\leq k\}$
	\end{center} is a $\Q-$linearly independent modulo $\C$ subset of $\mathbb{C}[[t_1,\ldots,t_m]]$ and that the $g_j$ for $1\leq j\leq k$ are algebraically independent over $\C$. Consider the fields $K_1=\mathbb{C}(\{f_i,g_{j},e^{f_i},e^{g_j}:1\leq i\leq n, 1\leq j\leq k\})$ and $K_2=\mathbb{C}(\{f_i,g_{j},e^{f_i}:1\leq i\leq n, 1\leq j\leq k\})$.
	
	Then by the classic Ax-Schanuel Theorem we have \begin{center}
		$tr.d._\mathbb{C}K_1\geq n+k+\rank(J(\vec{f},\vec{g};\vec{t}))$.
	\end{center}
	
	On the other hand \begin{center}
		$tr.d._{\C}\C(\{e^{g_j}:1\leq j\leq k\})+tr.d._{\C}K_2\geq tr.d._{\C}K_1$, and
		$tr.d_{\C}\C(\{e^{g_j}:1\leq j\leq k\})\leq k$.\end{center}Combining all three of these the result follows.\end{proof}
\begin{rmks}
	1. We note that \thref{asmainred} is probably known as a result in the field. However, since we couldn't find a reference, we have dedicated this short section to its proof.\\
	2. Following the ideas in \cite{tsimax} we can obtain the Full Ax-Schanuel Theorem for families along with a few corollaries. For these the interested readers are referred to the Appendix. 
\end{rmks}

We finally note, that the same proof gives, by reduction to the classic Ax-Schanuel Theorem\footnote{See Theorem $3$ in \cite{ax1971schanuel}.} as above, the following more abstract variant of the above proposition:
\begin{prop}\thlabel{propinter}Let $\Q\subset C\subset F$ be a tower of fields and $\Delta$ a set of derivations of the field $F$ with $\underset{D\in\Delta}{\bigcap}ker D=C$. Let $f_1,\ldots,f_m\in F$ and $z_1,\ldots,z_m\in F^*$ be such that for all $D\in \Delta$ and $1\leq i\leq m$, $Df_i=\frac{Dz_i}{z_i}$. Let us also assume that for some $n\leq m$ the $f_i$ for $1\leq i\leq n$ are $\Q-$linearly independent modulo $C$. Then,\begin{center}
		$tr.d._{C}C(\{ f_i,z_j :1\leq i \leq m, 1\leq j\leq n    \})\geq n+rank(Df_i)_{\underset{D\in\Delta}{1\leq i\leq m}}$.
	\end{center}

\end{prop}
\part{Upper Triangular Matrices}

\textbf{Notation:} We will denote by $U=U_n$ the algebraic group of $n\times n$ upper triangular invertible matrices over the field $\mathbb{C}$, and by $\mathfrak{h}=\mathfrak{h}_n$ its Lie algebra, i.e. the algebra of $n\times n$ upper triangular matrices. Also we denote the corresponding exponential map $\mathfrak{h}_n\rightarrow U_n$ by $E=E_n$ and its non-diagonal entries by $E_{i,j}$, $1\leq i<j\leq n$. 

We will mainly concern ourselves with transcendence degrees over the field $\C$ of extensions of the form $\C(\Sigma)$ with $\Sigma$ a finite subset of some ring of power series, or a finite subset of regular functions on some variety over the field $\C$. Of particular interest will be the case were $\Sigma$ is the set of entries of a matrix $A$, or a matrix $A$ and its exponential $E(A)$. In that case we denote the field extension $\C(\Sigma)$ over $\C$ by $\C(A)$ and $\C(A,E(A))$ respectively.

Consider elements $f_i,g_{i,j}\in\C[[t_1,\ldots,t_m]]$, with $1\leq i<j\leq n$. As in the introduction, we will denote by $\langle f_1,\ldots,f_n\rangle_{\Q}$ the linear span of the $f_i\in \C[[t_1,\ldots,t_m]]$ over $\Q$. Also, as in the introduction, $J(\vec{f},\vec{g};\vec{t})$ denotes the $\frac{n(n+1)}{2}\times m$ Jacobian matrix with entries of the form $\partial h_s/\partial t_j$, where $h_s$ with $1\leq s\leq \frac{n(n+1)}{2}$ is an ordering of the $g_{i,j}$ and the $f_i$. The rank of the Jacobian is its rank over the fraction field of $\mathbb{C}[[t_1,\ldots,t_m]]$. 

We also introduce some rings that will be needed in some of the proofs that follow. First consider $z_i$, $1\leq i\leq n$, and $x_{i,j}$, $1 \leq i<j\leq n$, to be independent variables over $\C$. Let $M_0:=\C(\{z_i,x_{i,j}:1\leq i<j\leq n\})$. Then for $1\leq i<j\leq n$, a subset $I\subset \{(s,t):1\leq s<t\leq n\}$, and a subset $J\subset\{1,\ldots,n\}$ we define the subrings \begin{center}
	$R_{i,j}(I):=\Q[x_{s,t},\frac{1}{z_l-z_j}: \min(i,j)\leq s<t\leq \max(s,t),(l,j)\notin I, (s,t)\neq (i,j)]$, and
\end{center} \begin{center}
	$D_{i,j}(J):=\Q(z_j)[\frac{1}{z_l-z_j},x_{s,t}: \min(i,j)\leq s<t\leq \max(i,j), i < l < j, l\notin J]$\end{center}of the field $M_0$.

We note that elements of $M_0$ can be naturally viewed as rational functions on $\mathfrak{h}_n$, with the $z_i(A)$ and $x_{s,t}(A)$ the corresponding entries of $A\in\mathfrak{h}_n$. Likewise, since they are subrings, the same is true for elements of the rings $D_{i,j}(J)$ and $R_{i,j}(I)$ defined above.\\

Our goal here is to state an Ax-Schanuel-type Theorem and reduce its proof to \thref{asmainred}. The first step towards that will be to find a lower bound for the transcendence degree \begin{center}
	$tr.d._\C \C(A,E(A))=tr.d._\C\C(\{f_i,g_{i,j}, E_{i,j}(A), e^{f_i}:1\leq i<j\leq n\})$
\end{center}where $f_i,g_{i,j}\in\C[[t_1,\ldots,t_m]]$, with $1\leq i<j\leq n$, and $A=\diag(\vec{f})+(g_{i,j})$.\\

We start with the case of $U_n$, instead of jumping straight to the case of $GL_n$, for a few reasons. First of all, the case of $U_n$, as we will see, is open to more calculations and, because of this, examples and notions, such as weakly special sets, can be more easily formulated in this setting. Furthermore, while the technical difficulties that appear in dealing with the exponential map seem to be of a similar nature in both of these linear groups, they are easier to deal with in the case of $U_n$, mainly again thanks to us being able to adopt a more computational approach. Finally, we believe that, in future work, the restrictions imposed by our method will be easier to lift in the case of $U_n$ first, so as to gain insight in the more technical case of $GL_n$. We return to this in \ref{applications}.

\section{Eigencoordinates }\label{tritosection}

In this section we consider fixed $f_i,g_{i,j}\in\C[[t_1,\ldots,t_m]]$ such that the $f_i$ are without constant terms, with $1\leq i<j\leq n$, and let $A=\diag(\vec{f})+(g_{i,j})$ be the matrix they define.

The main idea is that eigenvectors and generalized eigenvectors for a matrix $A$ will remain as such for the matrix $E(A)$. As we will see shortly the other information, that will naturally appear, and that we will have to keep track of, are the nilpotent operators defined by $A$ and $E(A)$ on their respective generalized eigenspaces.

In this section we define a canonically chosen basis of generalized eigenvectors for a given matrix $A\in \mathfrak{h}_n$ that will only depend on the multiplicities of the eigenvalues of $A$. To this basis we can assign coordinates, which will be rational functions on the entries of $A$. At the same time we achieve a canonical description of the respective nilpotent operators defined by $A$ and $E(A)$ on each of their generalized eigenspaces. To each such operator we will be able to naturally assign certain rational functions of the entries of $A$. The combination of the above rational functions, both those describing the basis and those describing the nilpotent operators, will be what we will refer to as \textbf{eigencoordinates}. 

These new notions have a distinct advantage, as we will see, in our setting, when dealing with questions surrounding transcendence properties. Namely they will allow us to:\\

1. replace the $g_{i,j}$ by the eigencoordinates of $A$, when dealing with transcendence questions, and\\

2. capture the essence of the map $E$, as far as transcendence is concerned, and replace the $E_{i,j}$ by the eigencoordinates of $A$, again in questions concerning transcendence.\\

The scope of this section is to state and prove the main lemmas that we will need concerning these new notions. As a motivation we first deal with the case where all of the eigenvalues $f_i$ of our matrix $A$ are distinct. After that we proceed with dealing with the general case.

\subsection{Distinct Eigenvalues}

Let $A$, $f_i$, and $g_{i,j}$ be as defined above, we assume that all the eigenvalues of our matrix $A$ are distinct. This is equivalent to the eigenvalues being distinct for both $A$ and $E(A)$, since the $f_i$ have no constant term. Among all the possible bases of eigenvectors for $A$ we choose one in a canonical way.

Let $K$ be an algebraic closure of the field $\C(\{f_i,g_{i,j}:1\leq i<j\leq n\})$, and let $t_{i,j}\in K$, with $1\leq i<j\leq n$ be such that the vector \begin{center}
	$\vec{v}_i=(-t_{1,i},\ldots,-t_{i-1,i},1,0,\ldots,0)$
\end{center} is an $f_i-$eigenvector for $A$. In this case $\vec{v}_i$ will also be an $e^{f_i}-$eigenvector for $E(A)$. We leave the proof of the existence of this canonical basis for $A$, chosen as above, to \thref{canbas}.

For the above, we will have the following
\begin{lemma}\thlabel{eigen1}Let $A$ and $t_{i,j}$ be as above. Then\begin{equation*}
tr.d._\C \C(A) =tr.d._\C \C(\{f_i,t_{i,j}:1\leq i<j\leq n   \}).
	\end{equation*} \end{lemma}
\begin{proof} The condition $A\vec{v}_j=f_j\vec{v}_j$ translates to the following system of equations
	\begin{gather}
	\begin{aligned}
	\frac{g_{i,j}-g_{i,j-1}t_{j-1,j}-\ldots -g_{i,i+1}t_{i+1,j}}{f_i-f_j}&=t_{i,j}\\
	\vdots\\\frac{g_{j-1,j}}{f_{j-1}-f_j}&= t_{j-1,j}
	\end{aligned}\label{eq:eq1}
	\end{gather}
Since all of the $f_i$ are distinct we can write the $t_{i,j}$ as rational functions on the entries of $A$ by solving the above system of equations. In particular we get that for all $i<j$: 
	
a. $t_{i,j}\in \Z[\frac{1}{f_s-f_j},g_{s,r}:\min(i,j)\leq s<r\leq \max(i,j)]$, and\\

b. Let $I_{i,j}=\{(s,j): s\leq i$ or $ s\geq j \} $. There exists $Q_{i,j}\in R_{i,j}(I_{i,j})$ such that\begin{center}
	$t_{i,j}(A)=\frac{g_{i,j}-Q_{i,j}(A)}{f_i-f_j}$.\end{center}
Our result now follows trivially from the above remarks.\end{proof}	

Since the matrices $A$ and $E(A)$ have the same eigenvectors, in the case where both $A$ and $E(A)$ have distinct eigenvalues, we get that, for all $1\leq i<j\leq n$, we will have $t_{i,j}(A)=t_{i,j}(E(A))$. This remark, together with the proof of \thref{eigen1}, implies

\begin{lemma}\thlabel{eigenfree}Let $A$ and $t_{i,j}$ be as defined above. Then \begin{equation*}
tr.d._\C \C(E(A)) =tr.d._\C \C(\{e^{f_i},t_{i,j}(A):1\leq i<j\leq n\}),\text{ and}\end{equation*}
\begin{equation*}tr.d._\C\C(A, E(A))=tr.d._\C\C(\{f_i,e^{f_i},t_{i,j}(A):1\leq i<j\leq n\}). \end{equation*}	
\end{lemma}
\begin{proof}
 Let $\vec{v}_i$ be the canonically chosen basis for $A$. Then, for all $1\leq i<j\leq n$, by combining the proof of \thref{eigen1} and the equality $t_{i,j}(A)= t_{i,j}(E(A))$, there exists $Q_{i,j}$ as in the proof of \thref{eigen1} such that \begin{center}$E_{i,j}(A)=(g_{i,j}-Q_{i,j}(A))\frac{e^{f_i}-e^{f_j}}{f_i-f_j}+Q_{i,j}(E(A))$.\end{center} The result then follows trivially from these remarks.\end{proof}

\subsection{Repeating eigenvalues}

In the case where we have eigenvalues with multiplicity greater than $1$ we have to alter our approach. The idea is to generalize the approach of the previous subsection. In other words, we wish to find a canonically defined basis, which will allow us to define coordinates that characterize our original matrix $A$ uniquely. Furthermore we wish to describe those coordinates as algebraic functions of the coordinates of $A$. Our ultimate goal is to obtain results about transcendence degrees similar to those we proved in the previous case.\\

We assume that the matrix $A$ has eigenvalues with multiplicities possibly greater than $1$. By our assumption that the $f_i$ have no constant term, each eigenvalue $e^{f_i}$ of $E(A)$ has the same multiplicity as the respective eigenvalue $f_i$ of $A$.

\subsubsection{The Canonical Basis}
Our first objective will be to describe and prove the existence of a certain canonical basis for $A$. We begin by describing the canonical basis of each eigenspace, then we combine these to create the basis we want. 

Just as before we let $A=\diag(\vec{f})+(g_{i,j})$ be an upper triangular matrix with entries in $\C[[t_1,\ldots, t_m]]$ and $K$ be an algebraic closure of the field $\C(\{f_i,g_{i,j}:1\leq i<j\leq n\})$.

\begin{lemma}\thlabel{canbas} Let $z=f_{i_1}=\ldots=f_{i_{k-1}}=f_{i_k}$ with $i_1<\ldots < i_k$. We also assume that $f_i\neq z$ for all $i\neq i_j$. Let $M(z)$ be the generalized eigenspace for the eigenvalue $z$. Then there exists a unique basis $\cal{B}$$_z$ of $M(z)$ consisting of vectors $\vec{v}_{i_j}$, $1\leq j\leq k$, such that 
\begin{enumerate}
	\item $\vec{v}_{i_j}$ is of the form \begin{center}
		$\vec{v}_{i_j}=(-t_{1,i_j},\ldots, -t_{i_j-1,i_j},1,0,\ldots,0)$, and
	\end{center}
\item $t_{l,i_j}=0$ for $l=i_r$ where $1\leq r\leq j-1$, and
\item there exist $s_{i_l,i_r}\in K$ for $1\leq l<r\leq k$ such that \begin{equation*}A\vec{v}_{i_j}=z\vec{v}_{i_j}+ s_{i_1,i_{j}}\vec{v}_{i_1}+\cdots+s_{i_{j-1},i_j}\vec{v}_{i_{j-1}}.
\end{equation*}
\end{enumerate}\end{lemma}
\begin{proof}We proceed by induction on $k=\dim_\C(M(z))$. For $k=1$ the uniqueness follows from the unique solution to the $t_{i,j}$ described by the equations \eqref{eq:eq1}. 

Assume that $k=2$. Then applying the system \eqref{eq:eq1} we can determine the vector $\vec{v}_{i_1}$ which will be an eigenvector for $A$. Since $\vec{v}_{i_2}$ will in general be a generalized eigenvector then we will have 
\begin{center}$(A-zI_n)\vec{v}_{i_2}=s_{i_1,i_2}\vec{v}_{i_1}$,\end{center}for some $s_{i_1,i_2}\in K$.
We can therefore assume without loss of generality that $t_{i_1,i_2}=0$.

This relation will describe the coefficients of $\vec{v}_{i_2}$ uniquely thanks to the following series of equations:
\begin{enumerate}
	\item In the range $i_1< j < i_2$ we get 
	\begin{gather}
	\begin{aligned}
	0&=-(f_{i_2-1}-f_{i_2})t_{i_2-1,i_2}+g_{i_2-1,i_2}\\
	&\vdots\\ 0&=-(f_{i_1+1}-f_{i_2})t_{i_1+1,i_2}-\ldots-g_{i_1+1,i_2-1}t_{i_2-1,i_2} +g_{i_1+1,i_2}
	\end{aligned}\label{eq:eq2}
	\end{gather}
	
	\item For $j=i_1$
	\begin{gather}
	\begin{aligned}
	   -g_{i_1,i_1+1}t_{i_1+1,i_2}-\ldots- g_{i_1,i_2-1}t_{i_2-1,i_2}+g_{i_1,i_2}&=s_{i_1,i_2}
\end{aligned}\label{eq:eq3}
\end{gather}

	\item In the range $1\leq j<i_1$
\begin{gather}
	\begin{aligned}
	-(f_{j}-f_{i_2})t_{j,i_2}-\ldots- g_{j,i_2-1}t_{i_2-1,i_2}+g_{j,i_2}&=s_{i_1,i_2}t_{j,i_1}.
\end{aligned}\label{eq:eq4}
\end{gather}
\end{enumerate}

Solving the above system, starting from the first equation of \eqref{eq:eq2} and moving to the final equation described in the system \eqref{eq:eq4} provides a unique solution for $t_{j,i_2}$ in terms of the $f_s$ and $g_{s,t}$.

Assume the result holds for $\dim_{\C} M(z)=k$. Then in order to prove the inductive step we can create a similar system of equations with unique solution for the $t_{i_s,i_t}$ in terms of the coefficient of the matrix $A$.

We force relations on the canonical basis to be chosen so that
\begin{equation}
A\vec{v}_{i_j}=z\vec{v}_{i_j}+ s_{i_1,i_{j}}\vec{v}_{i_1}+\cdots+s_{i_{j-1},i_j}\vec{v}_{i_{j-1}}\label{eq:eq6}.
\end{equation}

Then by induction it is enough to determine the $t_{s,i_{k+1}}$, since the rest of the vectors will constitute a basis for the respective eigenspace of a smaller diagonal submatrix of $A$. To do this we just translate \eqref{eq:eq6} for $j=k+1$ to a system of equations similar to the systems \eqref{eq:eq2}, \eqref{eq:eq3}, and \eqref{eq:eq4}. \end{proof}

Combining all of the canonical bases of the eigenspaces we get a basis 
\begin{center}
		$\vec{v}_i=(-t_{1,i},\ldots,-t_{i-1,i},1,0,\ldots,0)$,
\end{center}with $1\leq i\leq n$, $t_{i,j}\in K$ for $1\leq  i<j\leq n$, and such that $t_{i,j}=0$ if $f_i=f_j$. This will be the \textbf{canonical basis} of $A$.

\subsubsection{Basic Lemmas}

In \thref{canbas} we introduced the coefficients $s_{i_l,i_r}$ for $l<r$. These coefficients determine uniquely the nilpotent operator defined by $A$ on the generalized eigenspace $M(z)$, i.e. they determine the nilpotent operator $(A-zI_n)|_{M(z)}$. In particular if $z=f_{i_1}=\ldots=f_{i_k}$ then we have $s_{i_l,i_r}\in K$, where $K$ is an algebraic closure of $\C(A)$, and they are such that for $1\leq l<r\leq k$ \begin{equation}\label{eq:system}A\vec{v}_{i_j}=z\vec{v}_{i_j}+ s_{i_1,i_{j}}\vec{v}_{i_1}+\cdots+s_{i_{j-1},i_j}\vec{v}_{i_{j-1}}.
\end{equation}

From the proof of \thref{canbas} both the $t_{i,j}$ and the $s_{i,j}$ are rational functions of the $f_i$ and $g_{i,j}$. Their combined information turns out to be exactly what we will need in what follows.
\begin{defn} Let $A=\diag(\vec{f})+(g_{i,j})$ be an upper triangular matrix with entries in $\C[[t_1,\ldots, t_m]]$, such the diagonal entries $f_i$ have no constant term. Let $\vec{v}_i=(-t_{1,i},\ldots,-t_{i-1,i},1,0,\ldots,0)$ be a canonical basis for $A$. Also we consider the $s_{i,j}$ that satisfy the equations in \eqref{eq:system} for those eigenvalues of $A$ with multiplicity greater than $1$. We define the \textbf{eigencoordinates} of $A$ to be\begin{center}
		$T_{i,j}(A)=\begin{cases}
		t_{i,j},& \text{if } f_i\neq f_j\\
		s_{i,j},& \text{if } f_i=f_j
		\end{cases}$.
\end{center} \end{defn}

At this point we want to replicate the results of \thref{eigen1} and \thref{eigenfree}. We start with the following
\begin{lemma}\thlabel{ecoord1} Let $A$ be as above and let $T_{i,j}(A)$ be its eigencoordinates. Then\begin{equation*}
	tr.d._\C \C(A) =tr.d._\C \C(\{f_i,T_{i,j}(A):1\leq i<j\leq n   \}).
	\end{equation*}   \end{lemma}

\begin{proof} Consider the set $I_{i,j}(A)=\{(s,j): s\leq i$ or $ s\geq j,$ and $f_s\neq f_j  \}$. For notational convenience let us define the ring\begin{center}
		$R_{i,j}(A):= R_{i,j}(I_{i,j}(A))$.
	\end{center}
	
	From the proof of \thref{canbas} it follows that the $T_{i,j}$ are rational functions on the coordinates of the matrix such that\\
	
	(a) if $f_i\neq f_j$ then there exists a $Q_{i,j}\in R_{i,j}(A)$  such that \begin{center}$T_{i,j}(A)= \frac{g_{i,j}-Q_{i,j}(A)}{f_i-f_j}$.\end{center}
	
	(b) if $f_i=f_j$ then there exists $Q_{i,j}\in R_{i,j}(A)$ such that \begin{center}
		$T_{i,j}(A) = g_{i,j}-Q_{i,j}(A)$.\end{center}
	In other words the map $A\mapsto \diag(f_1,\ldots,f_n)+(T_{i,j}(A))$ is bijective and the coordinates are rational functions on the entries of $A$ with only factors of the form $f_i-f_j$ appearing in the denominator.
	
	The equality of the transcendence degrees in question then follows easily from the above remarks.\end{proof}

\begin{rmk}
	The proof of the above lemma actually gives us more. Namely from the proof it follows that \begin{center}
		$\C(A)=\C(\{f_i,T_{i,j}(A):1\leq i<j\leq n \})$.
	\end{center}
\end{rmk}

The next step here is to study the effects of the exponential function on the eigencoordinates. We record the main such results we will need in the following
\begin{prop}\thlabel{ecoord2} Let $A$ be as above and let $T_{i,j}(A)$ be its eigencoordinates. Then\begin{equation*}
	tr.d._\C \C(E(A)) =tr.d._\C \C(\{e^{f_i}, T_{i,j}(A):1\leq i<j\leq n\}).
	\end{equation*}
	From this and \thref{ecoord1} we conclude that\begin{equation*}tr.d._\C\C(A, E(A))=tr.d._\C\C(\{f_i,e^{f_i},T_{i,j}(A):1\leq i<j\leq n\}). \end{equation*}	   \end{prop}

\begin{proof}\renewcommand{\qedsymbol}{}We start by fixing some notation. We let $R_{i,j}(A)$ be the rings defined in the proof of \thref{ecoord1}.

Since we have already dealt with the case where all eigenvalues are distinct, we only have to study the behaviour of the exponential with respect to the generalized eigenspaces of dimension greater than $1$. 
	
We assume that $z=f_{i_1}=\ldots\ f_{i_k}$ with $f_i\neq z$ if $i\notin \{i_1,\ldots,i_k\}$ so that the system \eqref{eq:system} actually describes the eigencoordinates of $A$. We start by looking at the effect of the exponential on the equations of \eqref{eq:system}. 
	
By induction and the definition of the exponential we get the following relation for the exponential matrix $E(A)$\begin{center}
		$E(A)\vec{v}_{i_j}=e^z\vec{v}_{i_j}+ e^z[s_{i_1,i_{j}}\vec{v}_{i_1}+\cdots+s_{i_{j-1},i_j}\vec{v}_{i_{j-1}}] +$
		\begin{equation}+e^z[S_{i_1,i_{j}}\vec{v}_{i_1}+\cdots+S_{i_{j-2},i_j}\vec{v}_{i_{j-2}}]\label{eq:eq7}
		\end{equation}
	\end{center}
	where $S_{i_t,i_j}:= \Sum{t<l_1<\ldots< l_m<j}{}\frac{1}{m!}s_{i_t,i_{l_1}}\cdots s_{i_{l_{k-1}},i_{l_k}}$.

Most importantly \eqref{eq:eq7} implies that 
\begin{equation}\label{eq:rel}T_{i,j}(E(A))=\begin{cases}e^z[T_{i,j}(A)+S_{i,j}(A)]& \text{if } f_i=f_j\\
T_{i,j}(A)& \text{if } f_i\neq f_j \end{cases},\end{equation} 
where the $S_{i,j}$ will again be elements of the ring $R_{i,j}$, due to the proof of \thref{ecoord1}, and can therefore be considered as functions on $A$. 

Assuming that $f_i=f_j$ we define $J_{i,j}(A)=\{i_1,\ldots,i_k\}$ and we also define the ring\begin{center}
	$D_{i,j}(A):=D_{i,j}(J_{i,j}(A))$.\end{center}

\textbf{Claim:} Let $i,j$ be such that $f_i=f_j$. Then there exists $P_{i,j}\in D_{i,j}(A)$ such that \begin{equation}
T_{i,j}(E(A))=e^{z}(T_{i,j}(A))+P_{i,j}(E(A)).\label{eq:sxesh}
\end{equation} 

Assuming this claim we go about proving that the transcendence degrees in question are in fact equal. First we define the following fields:\begin{center}
	 $K_1=\C(\{e^{f_i},T_{i,j}(A):1\leq i<j\leq n\})$, $K_0=\C(E(A))$, and
	 
$L=\C(\{E(A),T_{i,j}(E(A)),T_{i,j}(A):1\leq i<j\leq n\})$.
\end{center}
We then have from the equations \eqref{eq:rel}, for the case $f_i\neq f_j$, and \eqref{eq:sxesh}, for the case $f_i=f_j$, that the extension $L/K_0(\{T_{i,j}(E(A)):1\leq i<j\leq n\})$ is algebraic. From the proof of \thref{ecoord1}, applied to the matrix $E(A)$, we get that $K_0(\{T_{i,j}(E(A)):1\leq i<j\leq n\})=K_0$. So $tr.d._{\C}L=tr.d._{\C}K_0$.

On the other hand, again from the proof of \thref{ecoord1}, we have that $L=K_1(\{T_{i,j}(E(A)):1\leq i<j\leq n  \})$. While \eqref{eq:rel}, together with the definition of $S_{i,j}$, tells us that the extension $K_1(\{T_{i,j}(E(A)):1\leq i<j\leq n  \})/K_1$ is algebraic. So that $tr.d._{\C}L=tr.d._{\C}K_1$ and the result follows.\end{proof}

\begin{proof}[Proof of the Claim] We assume we are in the same situation as above. Namely we assume that $z=f_{i_1}=\ldots\ f_{i_k}$ with $f_i\neq z$ if $i\notin \{i_1,\ldots,i_k\}$, so \eqref{eq:eq7} holds for all $1\leq j\leq k$. In particular, we need to show that, for all pairs $1\leq t< j\leq k$, there exists $P_{i_t,i_j}\in D_{i_t,i_j}(A)$ such that $T_{i_t,i_j}(E(A))=e^{z}(T_{i_t,i_j}(A))+P_{i_t,i_j}(E(A))$. By \eqref{eq:rel} it suffices to show the existence of a $P_{i_t,i_j}\in D_{i_t,i_j}(A)$ such that $S_{i_t,i_j}=P_{i_t,i_j}(E(A))$. From the definition of $S_{i,j}$ it suffices to prove that for all pairs $1\leq t< j\leq k$, there exists $F_{i_t,i_j}\in D_{i_t,i_j}(A)$ such that $s_{i_t,i_j}(A)=F_{i_t,i_j}(E(A))$. We prove this last assertion by induction on $j-t$.
	
Let us start with $j-t=1$. As a consequence of \eqref{eq:eq7} we have that $T_{i_{j-1},i_j}(E(A))=e^{z}T_{i_{j-1},i_j}(A)$ for all $j$. We can rewrite this as\begin{equation}\label{eq:eq11}s_{i_{j-1},i_j}(A)=e^{-z}s_{i_{j-1},i_j}(E(A)).\end{equation}The assertion now follows from the proof of \thref{ecoord1} applied to the matrix $E(A)$ that shows $s_{i_{j-1},i_j}(E(A))=E(A)_{i_{j-1},i_j}-Q_{i_{j-1},i_j}(E(A))$, where $Q_{i_{j-1},i_j}\in R_{i_{j-1},i_j}(A)$.

For $t=j-2$ the definition of $S_{i,j}$ and \eqref{eq:rel} imply\begin{center}
		$s_{i_{j-2},i_j}(E(A))= e^zs_{i_{j-2},i_j}(A)+e^zs_{i_{j-2},i_{j-1}}(A)s_{i_{j-1},i_j}(A)$.\end{center}This together with \eqref{eq:eq11} imply
	\begin{equation}\label{eq:eq12}
		s_{i_{j-2},i_j}(A)=e^{-z}(s_{i_{j-2},i_j}(E(A)))-e^{-2z}s_{i_{j-2},i_{j-1}}(E(A))s_{i_{j-1},i_j}(E(A)).
	\end{equation}Once again the assertion follows as above from the proof of \thref{ecoord1}.

Assume the assertion for all pairs $(i_x,i_y)$ with $y-x\leq m$. Let $j=t+m+1$. Then, by definition of $S_{i,j}$ and the inductive hypothesis, we get that there exists $P_{i_t,i_j}\in D{i_t,i_j}(A)$ such that $S_{i_t,i_j}=P_{i_t,i_j}(E(A))$. 

From \eqref{eq:rel} we get that $s_{i_t,i_j}(A)=F_{i_t,i_j}(E(A))=e^{-z}(s_{i_t,i_j}(E(A))-P_{i_t,i_j}(E(A)))$. Once again by the proof of \thref{ecoord1} applied to $E(A)$ as above the assertion follows and the claim has been proven.\end{proof}
\section{Ax-Schanuel for $U_n$}
 
At this point we are able to state and prove a Weak Ax-Schanuel-type result for the map $E:\mathfrak{h}_n\rightarrow U_n$. We also record a corollary of our result, as well as an alternate geometric view in the spirit of \cite{pilafun}.

As we have been doing so far, we let $E_{i,j}(A)$ be the corresponding entry of the matrix $E(A)$ for $1\leq i<j\leq n$.

\begin{theorem}[Weak Ax-Schanuel for $U_n$]\thlabel{asupper} Let $f_1,\ldots,f_n,g_{i,j}\in \mathbb{C}[[t_1,\ldots,t_m]]$ be power series, where $1\leq i<j\leq n$. We assume that the $f_i$ do not have a constant term. Let $A$ be the $n\times n$ upper triangular matrix with diagonal $\vec{f}$ and the $(i,j)$ entry equal to $g_{i,j}$. Then, assuming that the $f_i$ are $\mathbb{Q}-$linearly independent,\begin{center}
		$tr.d._\mathbb{C}\mathbb{C}(A,E(A))\geq n+\rank(J(\vec{f},\vec{g};\vec{t}))$.\end{center}
\end{theorem}
\begin{proof}From \thref{eigenfree} we may replace the left hand side of the above inequality by $tr.d._\C\C(\{f_i,g_{i,j},e^{f_i}:1\leq i<j\leq n\})$. This reduces the proof to \thref{asmainred}, by giving the $g_{i,j}$ a new indexing $g_k$, $1\leq k\leq \frac{n(n-1)}{2}$.\end{proof}

Replacing \thref{eigenfree} with \thref{ecoord2} in the above proof yields the following
\begin{cor}\thlabel{corasupp}Let $f_1,\ldots,f_n,g_{i,j}\in \mathbb{C}[[t_1,\ldots,t_m]]$ be power series, where $1\leq i<j\leq n$. We assume that the $f_i$ do not have a constant term. Let $A$ be the $n\times n$ upper triangular matrix with diagonal $\vec{f}$ and the $(i,j)$ entry equal to $g_{i,j}$. Let also $N=\dim_\Q\langle f_1,\ldots,f_n\rangle_{\Q}$, then \begin{center}
	$tr.d._\mathbb{C}\mathbb{C}(A, E(A))\geq N+\rank(J(\vec{f},\vec{g};\vec{t}))$.\end{center}
\end{cor}

\textbf{An Alternate Formulation\\}

In the spirit of \cite{pilafun} we can give an alternate form of \thref{asupper}. This time the background is slightly changed. We let $V\subset \mathfrak{h}_n$ be an open subset and $X\subset V$ an irreducible complex analytic subvariety of $V$ with $dim_\C(X)=m$ such that $X$ contains the origin, and locally at $\vec{0}$ the coordinate functions, $f_1,\ldots,f_n$ and $g_{s,t}$ for $s<t$, are meromorphic functions on $X$.

For reasons of convenience, and in keeping a similar notation to the previous version, we let $A=(\vec{f})+(g_{s,t})$ denote the matrix corresponding to the coordinates $f_i$ and $g_{i,j}$.\\

\begin{theorem}[Weak Ax-Schanuel-Alternate Formulation]\thlabel{aslast}In the above context, if the $f_i$ are $\mathbb{Q}-$linearly independent modulo $\mathbb{C}$, then 
	\begin{center}$tr.d._\mathbb{C}\mathbb{C}(A,E(A))\geq n+\dim(X)$.\end{center}
\end{theorem}\begin{proof}
 We choose  $t_1,\ldots,t_m$ that are independent holomorphic coordinates on $X$ locally at $\vec{0}$ so that $\rank(J(\vec{f},\vec{g};\vec{t}))=dim_\C(X)$. Then this reduces to \thref{asupper}.
\end{proof}

Also similarly to above we can translate \thref{corasupp} in this context.
\begin{cor}\thlabel{corasupalt}In the above context, if the $f_i$ are are such that $N=\dim_\Q\langle f_1,$ $\ldots$ $,f_n\rangle_{\Q}$, then 
	\begin{center}$tr.d._\mathbb{C}\mathbb{C}(A,E(A))\geq N+\dim(X)$.\end{center}\end{cor}
This form of the Ax-Schanuel result is the one we will use in what follows. 

\section{Weakly Special Subvarieties for $\mathfrak{h}_n$}\label{wesp}

We turn our attention to describing the weakly special subvarities of $\mathfrak{h}_n$ that contain the origin, i.e. the zero matrix. Geometrically \thref{ecoord2} gives us a significant amount of motivation. We can expect that the weakly special subvarieties will be determined by the following information: \begin{enumerate}
	\item A system $\Sigma$ of $\Q-$linear equations on the diagonal coordinates of the matrices, i.e. the eigenvalues, and
	\item A system of equations on the eigen-coordinates of a generic matrix in the subvariety of $\mathfrak{h}_n$ defined by the system $\Sigma$.
\end{enumerate}

\subsection*{Conditions on Eigenvalues}

We start by making this idea more explicit. Let us assume that we have a $\Q-$linearly independent set $\Sigma$ of $\Q-$linear polynomials on the diagonal coordinates of $\mathfrak{h}_n$, i.e. polynomials of the form \begin{center}
	$F(\vec{f})=\Sum{i=1}{n}q_i f_i$,
\end{center}where $q_i\in \Q$. Let $Z(\Sigma)\subset \mathfrak{h}_n$ be the algebraic subvariety of $\mathfrak{h}_n$ defined by $\Sigma$. 

Picking a generic matrix $A\in Z(\Sigma)$ the multiplicities of the eigenvalues will be determined by $\Sigma$. More specifically, depending on $\Sigma$ we have fixed multiplicities of eigenvalues on a dense open subset of $Z(\Sigma)$, which we denote by $U_\Sigma$. To define this latter subset we start by considering the set \begin{center}
$I_\Sigma=\{(i,j): 1\leq i<j\leq n, Z(\Sigma)\subset Z(f_i-f_j) \} $.
\end{center}Then we take
\begin{center}
$U_\Sigma= Z(\Sigma)\backslash (\underset{(i,j)\notin I_\Sigma}{\bigcup} Z(f_i-f_j)) $.
\end{center}

\subsection*{Passing to the eigencoordinates}

Let us restrict our attention to $U_\Sigma$. On this dense open subset we can define, by \thref{canbas} a canonical basis for any matrix $A\in U_\Sigma$. The eigen-coordinates of $A$ will be well defined regular functions on $U_\Sigma$, thanks again to the proof of \thref{canbas}.

The relations proven during the proof of \thref{ecoord2} reinterpreted geometrically show that any algebraic relation satisfied by the $T_{i,j}(A)$ will translate to an algebraic relation for the $T_{i,j}(E(A))$ and vice versa. In order to translate this into a geometric language we must first find a more convenient description for $Z(\Sigma)$ and $U_\Sigma$.

We start by noting that we have an isomorphism \begin{center}
	$Z(\Sigma)\cong L\times \mathbb{A}_\C^{\frac{n(n-1)}{2}}$,
\end{center} where $L$ is a $\Q-$linear subspace of $\C^k$ where $k$ is the number of generically distinct eigenvalues of $Z(\Sigma)$. Let us also consider the dense open subset $D_\Sigma\subset \C^k$ with
\begin{center}
	$D_\Sigma=\C^k\backslash (\underset{1\leq i<j\leq k}{\bigcup}Z(x_i-x_j))$,
\end{center}where $x_i$ denote the coordinates of $\C^k$.

Combining the above, we consider $L'=L\cap D_\Sigma$ under this identification. We may then take the following isomorphism\begin{equation}
U_\Sigma\cong L'\times \mathbb{A}_\C^{\frac{n(n-1)}{2}}\label{eq:iso}.\end{equation}

At this point we apply \thref{ecoord1}, which shows that we can change coordinates on the $\mathbb{A}_\C^{\frac{n(n-1)}{2}}$ part of the right hand side of the above isomorphism from $g_{i,j}$ to $T_{i,j}$. In other words, we have an isomorphism\begin{equation}
	T_\Sigma:U_\Sigma\rightarrow L'\times \mathbb{A}_\C^{\frac{n(n-1)}{2}}\label{eq:iso2},
\end{equation}where the $\mathbb{A}_\C^{\frac{n(n-1)}{2}}$ on the right signifies the affine space of the eigencoordinates $T_{i,j}$.

\subsection*{Conditions on the eigencoordinates}

 Let $V\subset \mathbb{A}_\C^{\frac{n(n-1)}{2}}$ be an irreducible subvariety of $\mathbb{A}_\C^{\frac{n(n-1)}{2}}$ that contains the origin, where the latter is considered as the space of the eigencoordinates. Then if we consider $W=L'\times V$ this will be an irreducible subvariety of $L'\times \mathbb{A}_\C^{\frac{n(n-1)}{2}}$. We now consider its inverse under the isomorphism $T_\Sigma$ of \eqref{eq:iso2}. Finally we consider the Zariski closure of the resulting set in $\mathfrak{h}_n$, which we will denote by $X(\Sigma, V)$. 

Notice that $X(\Sigma,V)$ satisfies exactly what we wanted, the diagonal coordinates are only subject to $\Q-$linear equations, any relation on the strictly upper triangular part comes from relations on the eigencoordinates, it is irreducible and it contains the origin.
\begin{defn}[Weakly Special Subvarieties]An irreducible subvariety $X$ of $\mathfrak{h}_n$ that contains the origin will be called \textbf{weakly special} if there exist:\begin{enumerate}
		\item a system $\Sigma$ of $\Q-$linear equations on the diagonal entries, and
		\item an irreducible subvariety $V$ defined as above, 
	\end{enumerate}such that $X=X(\Sigma,V)$, where the latter is as defined in the above discussion.\end{defn}
\section{Ax-Lindemann for $U_n$ and other corollaries}\label{saxlindup}

Here we record some corollaries of our Ax-Schanuel result. We start with a ``two-sorted version"\footnote{We are borrowing this term from the relative discussion in \cite{pilafun}.} of \thref{corasupalt} and then use that to prove the Ax-Lindemann result. The latter allows us to characterize the bi-algebraic subsets for the map $E$ that contain the origin. The exposition follows in the spirit of \cite{pilafun}. 

We start by defining the notion of a component.
\begin{defn} Let $W\subset \mathfrak{h}_n$ and $V\subset U_n$ be algebraic subvarieties. Then a \textbf{component} $X$ of $W\cap E^{-1}(V)$ will be a complex-analytically irreducible component of $W\cap E^{-1}(V)$.\end{defn}

The context in which we will be using our Ax-Schanuel result is the one described in \thref{aslast} and the discussion leading up to it.

\begin{theorem}[Two-sorted Weak Ax-Schanuel for $U_n$]\thlabel{atypint} Let $U\subset \mathfrak{h}_n$ be a weakly special subvariety, containing the origin, and set $X=E(U)$. Let $W\subset U$ and $V\subset X$ be algebraic subvarieties, with $0_n\in W$ and $I_n\in V$. If the component $C$ of $W\cap E^{-1}(V)$ that contains the origin is not contained in any proper weakly special subvariety of $U$ then \begin{center}
		$\dim_\C C\leq \dim_\C V+\dim_\C W-\dim_\C X$.\end{center} \end{theorem}
\begin{proof} Following the discussion of the previous section, we can associate to the subvariety $U$ a system $\Sigma$ of $\Q-$linear equations on the diagonal entries, as well as the corresponding $\Q-$linear subspace $L$ of $\C^k$ where $k$ is the number of generically distinct eigenvalues, and a subvariety $Z$ of $\mathbb{A}_\C^{\frac{n(n-1)}{2}}$. In other words with the notation of the previous section $U=X(\Sigma,Z)$. We also denote by $U_\Sigma$  the corresponding dense open subset we had in the discussion of the previous section.

At this point we let $B=U_\Sigma \cap C$, then $B$ is again a complex analytically irreducible subset that is dense in $C$ and it is not contained in a proper weakly special subvariety of $U$. In particular we will have $\dim_\C B=\dim_\C C$.

We denote by $f_i$ the diagonal coordinates of a matrix as functions on $B$ and similarly for the coordinates $g_{i,j}$. Likewise we denote the diagonal coordinates of the exponential map by $E_i$ and the strictly upper triangular by $E_{i,j}$ and we consider them as functions on $B$ as well, keeping in mind that $E_i(A)=e^{f_i}$. For reasons of convenience we let $A=\diag(\vec{f})+(g_{i,j})$ denote the matrix of the corresponding coordinates.

We start with some simple remarks concerning our setting. First of all, we will have\begin{equation}\dim_\C W\geq tr.d._\C\C(A)\text{, and}\label{eq:dimw}\end{equation}\begin{equation}
\dim_\C V\geq tr.d._\C \C(E(A))\label{eq:dimv}.
\end{equation}	

Next, we employ \thref{corasupalt}, to get that, if $N=\dim_\Q \langle f_1,\ldots,f_n\rangle_{\Q}$, then\begin{equation}\dim_\C B+N\leq tr.d._\C \C(A,E(A)).\label{eq:axschan}\end{equation}
We also set \begin{center}
	$m=tr.d._\C \C(A,E(A))$, and
\end{center}\begin{center}
 $l=tr.d._\C \C(\{T_{i,j}(A):1\leq i<j\leq n\})$,
\end{center}with $T_{i,j}$ denoting once again the eigencoordinates of a matrix. From this point on for convenience we will denote simply by $T_{i,j}$ the elements $T_{i,j}(A)$.

At this point we turn our attention to \thref{ecoord1}, \thref{ecoord2}, along with equations \eqref{eq:rel} and \eqref{eq:sxesh}. On $B$ the eigencoordinates $T_{i,j}$ are well defined as functions on $B$. From the aforementioned lemmas we also get
\begin{equation}tr.d._\C \C(A)=tr.d._\C \C(\{T_{i,j},f_i:1\leq i<j\leq n\}), \text{ and}\end{equation}
\begin{equation}tr.d._\C \C(E(A))=tr.d._\C \C(\{T_{i,j},E_i:1\leq i<j\leq n\}).\end{equation}

By the definition of weakly special subvarieties we see that \begin{equation}
\dim_\C U=\dim_\C E(U)=\dim_\C X= N+l.\label{eq:dimx}
\end{equation} On the other hand, by the minimality of the weakly special subvariety $U$, we get that, if $K=\C(\{T_{i,j}:1\leq i<j\leq n\})$,
\begin{center}
	$m=l+tr.d._KK(\{E_i,f_i:1\leq i\leq n\})$
\end{center}
We also have that\begin{center}
	 $tr.d._\C \C(A)= tr.d._\C K+tr.d._K K(f_1,\ldots,f_n)$,
\end{center} and likewise that \begin{center}
	$tr.d._\C \C(E(A))= tr.d._\C K+ tr.d._K K(E_1,\ldots,E_n)$.
\end{center}

Combining the above equalities implies that 
\begin{equation}
m\leq tr.d._\C \C(E(A))+tr.d._\C \C(A)-l\label{eq:protel}.
\end{equation}
Using \eqref{eq:protel} along with \eqref{eq:dimv} and \eqref{eq:dimw} yields 
\begin{center}$m\leq \dim_\C W+\dim_\C V-l$.\end{center}
Together with \eqref{eq:axschan}, \eqref{eq:dimx}, and the fact that $\dim_\C B=\dim_\C C$, this finishes the proof.\end{proof}

\begin{cor}[Ax-Lindemann for $U_n$]\thlabel{axlind}Let $V\subset U_n$ be an algebraic subvariety with $I_n\in V$. If $W\subset E^{-1}(V)$ is a maximal irreducible subvariety that contains the origin, then $W$ is a weakly special subvariety.\end{cor}
\begin{proof}Let $U$ be the minimal weakly special subvariety that contains $W$, $X=E(U)$, and let $V'=V\cap X$. We use \thref{atypint} for $C=W$ to get \begin{center}
	$\dim_\C W\leq \dim_\C{W}+\dim_\C V'-\dim_\C X$.\end{center}This implies $\dim_\C X\leq \dim_\C V'$, and since $V'\subset X$ we get that $X\subset V$ and that $W\subset U\subset E^{-1}(V)$. Maximality of $W$ then implies that $W=U$ is weakly special.   \end{proof}
\part{General Matrices}
Having studied the exponential of $\mathfrak{h}_n$ we can expect to achieve similar Ax-Schanuel and Ax-Lindemann results for the case of general matrices. Once again the key role will be played by the eigenvalues of our matrix. 

We start with considering certain subsets of $\mathfrak{gl}_n$ that will assist us in formulating the Ax-Schanuel and Ax-Lindemann results. We then proceed in a similar fashion to the upper triangular case. Namely we start by stating the Ax-Schanuel result and then reduce its proof to \thref{asmainred}. Finally, we conclude with some corollaries of our result.\\

\textbf{Notation:} For the remainder we will denote the Lie algebra of $n\times n$ matrices over $\C$ by $\mathfrak{gl}_n$ and the respective exponential function by \begin{center}
	$E:\mathfrak{gl}_n\rightarrow GL_n$.
\end{center}

\section{Data of a matrix and the exponential}\label{wekspegln}

We begin our study by defining the \textbf{data} of a matrix $A$, a notion that will generalize the eigencoordinates we had in the upper triangular case. With the help of this new notion we can define, as we will see, the weakly special subvarieties and achieve a simpler description of the exponential. 

As we did in the case of the upper triangular matrices, throughout this section we present as lemmas the equalities of transcendence degrees that we will need in the proofs of our main results. 

\subsection{The Data of a matrix}\label{datasection}

Let $V$ be a $\C-$linear space with $\dim_\C V=n$. Let also $A\in \homm(V,V)=\mathfrak{gl}_n$ then $A$ is uniquely characterized by the following data:\begin{enumerate}
	\item A number of distinct complex numbers $f_1,\ldots,f_k$, the eigenvalues of $A$,
	\item for each eigenvalue $f_i$ an $m_i\in \N$, the multiplicity of that eigenvalue, such that $\Sum{i=1}{k}m_i=n$,
	\item for each $f_i$ as above, a subspace $V_i\leq V$, with $\dim V_i=m_i$, such that $V=\displaystyle\bigoplus_{i=1}^k V_i$, i.e. to every eigenvalue a corresponding generalized eigenspace, and
	\item for each $f_i$ as above, a nilpotent operator $N_i\in\homm(V_i,V_i)$, i.e. $N_i=(A-f_iI_n)|_{V_i}$.
\end{enumerate}The above picture also holds over an arbitrary algebraically closed field.

\begin{defn} Let $A$ be a matrix as above. Then we define the \textbf{data of the matrix} to be the data \begin{center}
		$((f_1,\ldots,f_k),(m_1,\ldots,m_k),(V_1,\ldots,V_k), (N_1,\ldots,N_k))$.
\end{center} \end{defn}
The information of the generalized eigenspaces $V_i$ and nilpotent operators $N_i$ of a matrix $A$ with $k$ distinct eigenvalues, each with respective multiplicity $m_i$, is parametrized by a variety which we will denote by $W_k(\vec{m})$. We also let $w_{k,\vec{m}}=\dim W_k(\vec{m})$. In what follows we will need to consider a set of coordinates on such a variety, which we will denote by $T_j$ with $1\leq j\leq w_{k,\vec{m}}$. 

These $T_j$ will play the role of the eigencoordinates of Part I. We digress here to properly define these varieties and make the above ideas more rigorous. We do this over $\C$, though the same construction clearly works over any algebraically closed field.

\subsubsection*{Some auxiliary varieties}

We fix an $n-$dimensional vector space $V$ over $\C$, we also fix $k\in\N$ and $m_1\leq\ldots\leq m_k\in\N$ such that $\Sum{i=1}{k}m_i=n$. We need a space parametrizing all pairs of $k-$tuples of the form $((V_1,\ldots, V_k),(N_1,\ldots,N_k))$ where $V_i$ is an $m_i-$dimensional subspace of $V$, $N_i$ is a nilpotent operator on $V_i$, and the $V_i$ are such that $V=\displaystyle\bigoplus_{i=1}^k V_i$.

To this end, consider the product of Grassmannians \begin{center}
	$G(\vec{m},V):=Gr_{m_1}(V)\times\ldots\times Gr_{m_k}(V)$.
\end{center}On this space we consider the trivial bundle $S:=V\times G(\vec{m},V)$ and for $1\leq i\leq k$ the subbudle $\mathbb{V}_i$ of $S$ that is the pullback of the tautological bundle of the Grassmannian $Gr_{m_i}(V)$ on $G(\vec{m},V)$.

Now consider the morphism of vector bundles over $G(\vec{m},V)$ \begin{center}
	$\phi_{(\vec{m},V)}:\bigwedge^{m_1}\mathbb{V}_1\wedge\ldots\wedge\bigwedge^{m_k}\mathbb{V}_k\rightarrow \bigwedge^{n} S $.
\end{center}The set $A(\vec{m},V):=\{P\in G(\vec{m},V):\phi_{(\vec{m},V),P}\neq 0  \}$ is an open subvariety of $G(\vec{m},V)$.

\begin{defn}
Let $X$ be a topological space and $F$ a finite dimensional vector bundle on $X$. Then we define $Nil(F)$ to be the vector bundle over $X$ whose fiber at $x\in X$ is the vector space $Nil(F_x)$ of nilpotent operators on $F_x$. 
\end{defn}

Let us now consider the vector bundle $E(\vec{m})=Nil(\mathbb{V}_1)\times\ldots\times Nil(\mathbb{V}_k)$ over $G(\vec{m},V)$. Then the restriction $E(\vec{m})|_{A(\vec{m},V)}$  of this vector bundle on $A(\vec{m},V)$ is exactly the space we want. So we define $W_k(\vec{m}):=E(\vec{m})|_{A(\vec{m},V)}$. 


\subsubsection*{Some auxiliary maps}

In what follows we will also need to consider a group action on $X_k(\vec{m})=\mathbb{A}_\C^k\times W_k(\vec{m})$. Consider the equivalence relation on $\{1,\ldots,k\}$ given by $i\sim j$ if and only if $m_i=m_j$. Let $i_1,\ldots,i_r$ be representatives for the equivalence classes of this equivalence relation, and for $1\leq j\leq r$ we let \begin{center}
	$n_j:=|\{s: 1\leq s\leq k, s\sim i_j   \}  |$
\end{center} and note that $\Sum{j=1}{r}n_j=k$. 

Let $S_{\vec{m}}:=S_{n_1}\times\ldots\times S_{n_r}$ be the direct product of the symmetric groups $S_{n_j}$. Each group $S_{n_j}$ acts naturally as permutations on $\mathbb{A}_\C^{n_j}$ and again as permutations of the factors $Gr_{m_s}(V)$ with $s\sim i_j$ of $G(\vec{m},V)$. As a result we get a natural action of each $S_{n_j}$ on $E(\vec{m})$ and by restriction on $W_k(\vec{m})$.

Putting all of these actions together we get an action of $S_{\vec{m}}$ on $\mathbb{A}_\C^{k}$ and one on $W_k(\vec{m})$. Because of our convention that $m_1\leq \ldots\leq m_k$, we may assume that the $S_{n_1}-$factor of $S_{\vec{m}}$ acts on the first $n_1$ coordinates of $\mathbb{A}_\C^{k}$, the $S_{n_2}-$factor on the next $n_2$ coordinates and so on. These two actions of $S_{\vec{m}}$ combine to give a diagonal action on $X_k(\vec{m})$.

If we let $f_1,\ldots,f_k$ be coordinates on $\mathbb{A}_\C^{k}$ we define $\Gamma_k:=\mathbb{A}_\C^{k}\backslash \underset{i<j}{\bigcup}Z(f_i-f_j)$ and define $Y_k(\vec{m}):= (\Gamma_k\times W_k(\vec{m}))/S_{\vec{m}}$. We also have a finite surjective morphism\footnote{See Theorem 1, pages 104-105 of \cite{mumfordav}.} that we denote by \begin{center}
	$\pi_{k,\vec{m}}:\Gamma_k\times W_k(\vec{m})\rightarrow Y_k(\vec{m})$.
\end{center}

For coordinates $f_i$, $1\leq i\leq k$, of $\mathbb{A}_{\C}^k$ and coordinates $T_j$, $1\leq j\leq w_{k,\vec{m}}$ of $W_k(\vec{m})$ we denote $[(\overrightarrow{(f_i(P))},\overrightarrow{(T_j(P))})] \in Y_k(\vec{m})$ the image of the point $P=(\overrightarrow{(f_i(P))},\overrightarrow{(T_j(P))})\in X_k(\vec{m})$. We also define the map \begin{center}
	$\Psi_{k,\vec{m}}:Y_k(\vec{m}) \rightarrow   \mathfrak{gl}_n$
\end{center}such that $\{[(\overrightarrow{(f_i(P))},\overrightarrow{(T_j(P))})]\}\mapsto QJQ^{-1}$, with $J$ being a block diagonal matrix, with its blocks being Jordan blocks, where we allow elements of the superdiagonal to be either $0$ or $1$, and $Q$ being the transition matrix that is defined by the subspaces $V_i$ parametrized by the coordinates $T_j$.

We also define \begin{center}
	$\Phi_{k,\vec{m}}:=\Psi_{k,\vec{m}}\circ\pi_{k,\vec{m}}:\Gamma_k\times W_k(\vec{m})\rightarrow \mathfrak{gl}_n$.
\end{center}

\begin{rmks}1. The image of $\Psi_{k,\vec{m}}$ will be exactly the set of all matrices with $k$ distinct eigenvalues whose multiplicities are given by the entries of the vector $\vec{m}$. This is true since the Jordan canonical form of a matrix is uniquely determined, up to permutation, by the Jordan blocks. Permuting these blocks also results in respective permutations of the columns of the transition matrix $P$, which are parametrized by the $T_j$.
	
2. We note that while $\Psi_{k,\vec{m}}$ is injective, $\Phi_{k,\vec{m}}$ is not. Nevertheless, it is a quasi-finite morphism of varieties, since all of its fibers are finite.

3. The map $\pi_{k,\vec{m}}$ is \'etale, since the action of $S_{\vec{m}}$ is free. Since $\Phi_{k,\vec{m}}$ is an open immersion onto its image, $\Psi_{k,\vec{m}}$ and $\Phi_{k,\vec{m}}$ are also \'etale onto their image.

4. The group action that we defined above reflects a new level of geometric complexity to the case of $\mathfrak{gl}_n$ compared to that for $\mathfrak{h}_n$. This stems from the fact that in $\mathfrak{gl}_n$ there is no a priori order to the eigenvalues of a matrix, in contrast to what happens in $\mathfrak{h}_n$, where they are naturally ordered in the diagonal.\end{rmks}

\subsubsection{Changing coordinates}

What is most important in our setting is that the passage from a matrix to its data preserves the transcendence degree. As in the upper triangular case, we start by considering elements in some ring of formal power series.\\

Let  $g_{i,j}\in\C[[t_1,\ldots,t_l]]$, $1\leq i,j\leq n$ and write $A=(g_{i,j})$. Then the eigenvalues of $A$ are elements of the integral closure of $\C[[t_1,\ldots,t_l]]$. By the Newton-Puiseux Theorem we know that this is contained in the field \begin{center}
	$\mathbb{L}=\underset{\vec{r}\in\N^l}{\bigcup} \C((t_1^{\frac{1}{r_1}} ))\cdots((t_l^{\frac{1}{r_l}}))$.
\end{center}

Assume that there are exactly $k$ distinct such eigenvalues $f_i$ of $A$, and that they have corresponding multiplicities $m_i$. Then the coordinates $T_j(A)$ of the point of $W_k(\vec{m})$ are also elements of the field $\mathbb{L}$.

We start by making rigorous the fact that changing from coordinates of $\mathfrak{gl}_n$ to coordinates of the data does not affect the transcendence degree.
\begin{lemma}\thlabel{trdata}Let  $g_{i,j}\in\C[[t_1,\ldots,t_l]]$, $1\leq i,j\leq n$. Let $K$ be an algebraic closure of the field $\C(A)$ and assume that the matrix $A=(g_{i,j})$ has exactly $k$ distinct eigenvalues $f_1,\ldots,f_k\in K$ with respective multiplicities $m_i$. Let also $T_j=T_j(A)$, $1\leq j\leq w_{k,\vec{m}}$, be the coordinates of the point in the variety $W_k(\vec{m})$ parametrizing the rest of the corresponding data of $A$. Then
	\begin{equation*}
	tr.d._\C \C(A)=tr.d._\C \C (\{f_i,T_j:1\leq i\leq k, 1\leq j\leq w_{k,\vec{m}}\}).
	\end{equation*}\end{lemma} 
\begin{proof} Let $K_0=\C (\{f_i,T_j:1\leq i\leq k, 1\leq j\leq w_{k,\vec{m}}\})$ and  \begin{center}
		$K_1=\C(\{g_{i,j}, T_r, f_s:1\leq i,j\leq n, 1\leq s\leq k, 1\leq r\leq w_{k,\vec{m}} \})$.
	\end{center}
The $T_j$ and $f_i$ are algebraic over $\C(A)$. So $tr.d._\C \C(A)=tr.d._\C K_1$.
	
On the other hand the $T_j$ and $f_i$ determine the $g_{i,j}$ via an algebraic process, in particular by determining the Jordan canonical form and the transition matrix $Q$ above. So we get that $tr.d._\C K_0=tr.d._\C K_1$.

Combining the two equalities of transcendence degrees the result follows. \end{proof}

\subsection{The Exponential Map}

Let $A=(g_{i,j})$ be a matrix with $g_{i,j}\in \C[[t_1,\ldots,t_l]]$, where the $g_{i,j}$ have no constant term. Let us also assume that $A$ has data given by\begin{center}
		$((f_1,\ldots,f_k),(m_1,\ldots,m_k),(V_1,\ldots,V_k), (N_1,\ldots,N_k))$.
\end{center}We are able to consider such data working over an algebraic closure $K$ of the field $\C(A)$. We would like to extract from this a simpler way of computing the effect of the exponential on $A$. 

Since the $g_{i,j}$ have no constant term, it is easy to see that the distinct $f_i$ cannot differ by an integral multiple of $2\pi i$, and hence the corresponding data for the matrix $E(A)$ will be:\begin{enumerate}
	\item the eigenvalues will be the distinct elements $e^{f_1},\ldots, e^{f_k}\in\mathbb{L}$, 
	\item the multiplicities $m_i$ will be the same,
	\item the generalized eigenspaces $V_i$ will remain as such, and
	\item the nilpotent operator corresponding to $e^{f_i}$ is \begin{center}
		$N_i'= e^{f_i} (E(N_i)-id_{V_i})$.
	\end{center}
\end{enumerate}

Let $E_{i,j}(A)$ denote the $(i,j)$ entry of the exponential matrix $E(A)$. Then we will have the following
\begin{prop}\thlabel{trdataexp}Let  $g_{i,j}\in\C[[t_1,\ldots,t_l]]$, $1\leq i,j\leq n$, be such that the $g_{i,j}$ have no constant term. We assume that $A=(g_{i,j})$ has exactly $k$ distinct eigenvalues $f_1,\ldots,f_k$ with respective multiplicities $m_1,\ldots,m_k$. Let $K$ be an algebraic closure of the field $\C(A)$ and $T_j=T_j(A)$, $1\leq j\leq w_{k,\vec{m}}$, be coordinates for the variety $W_k(\vec{m})$ parametrizing the corresponding data of $A$. Then
\begin{equation*}
tr.d._\C \C(E(A))=tr.d._\C \C (\{e^{f_i},T_j:1\leq i\leq k, 1\leq j\leq w_{k,\vec{m}}\}), \text{ and}
\end{equation*}\begin{equation*}
tr.d._\C \C(A,E(A))=tr.d._\C \C (\{f_i,e^{f_i},T_j:1\leq i\leq k, 1\leq j\leq w_{k,\vec{m}}\}).
\end{equation*}\end{prop} 
\begin{proof}\renewcommand{\qedsymbol}{}Let $\bar{T}_j=T_j(E(A))$, $1\leq j\leq w_{k,\vec{m}}$, be the coordinates in $W_k(\vec{m})$ parametrizing the corresponding data of $E(A)$. From \thref{trdata} applied to the matrix $E(A)$ we get
\begin{equation*}
tr.d._\C \C(E(A))=tr.d._\C \C (\{e^{f_i},\bar{T}_j:1\leq i\leq k, 1\leq j\leq w_{k,\vec{m}}\}).\end{equation*}

Therefore we are left with proving the following equality\begin{equation*}\begin{split}
tr.d._\C \C(\{e^{f_i},\bar{T}_j:1\leq i\leq k, 1\leq j\leq w_{k,\vec{m}} \})=& \\
=tr.d._\C \C (\{e^{f_i},T_j:1\leq i\leq k, 1\leq j\leq w_{k,\vec{m}}\}).&\end{split}\end{equation*}

By the remarks above though the $T_j$ and $\bar{T}_j$ will parametrize the same $V_i$, so that their only difference is located in those $\bar{T}_j$ and $T_j$ that parametrize the nilpotent operators. For the latter we know that we will have\begin{center}
	$N_i'= e^{f_i} (E(N_i)-id_{V_i})$.\end{center}

\textbf{Claim:} The map $f:Nil(m)\rightarrow Nil(m)$ given by $N\mapsto E(N)-id$ is a bialgebraic map, where $Nil(m)$ denotes the space of nilpotent operators on an $m-$dimensional $\C-$vector space.\\

Assuming this claim holds, if $T^i_{j}$, $\bar{T}^i_{j}$, $j\in J_i$ denote the elements among the $T_j$, and $\bar{T}_j$ respectively, that parametrize the information of the nilpotent operators $N'_i$ and $N_i$, then the above shows that 
\begin{center}
	$\C(\{ e^{f_i},T^i_j:j\in J_i \})=\C(\{e^{f_i}, \bar{T}^i_j:j\in J_i \})$,
\end{center}for all $i=1,\ldots,k$. Combining this with the fact that $T_j=\bar{T}_j$ for all of the rest, i.e. those parametrizing the $V_i$, the result follows trivially.

The above argument shows that in fact $\C(\{ e^{f_i},T_j:1\leq j\leq w_{k,\vec{m}}\})=\C(\{e^{f_i}, \bar{T}_j:1\leq j\leq w_{k,\vec{m}} \})$. Combining this with the remark at the end of the proof of \thref{trdata} we get that the field $F_1=\C(\{A,e^{f_i},T_j:1\leq i\leq k,1\leq j\leq w_{k,\vec{m}}\})$ is a finite algebraic extension of the field $\C(A,E(A))$. Similarly, $F_2=\C(\{f_i,e^{f_i},T_j:1\leq i\leq k, 1\leq j\leq w_{k,\vec{m}}\})$ is a finite algebraic extension of $F_1$ which finishes the proof of the second equality.\end{proof}

\begin{proof}[Proof of the Claim] Let $N$ be a nilpotent operator on an $m-$dimensional vector space and let $k\in\N$ be such that $N^k\neq 0$ and $N^{k+1}=0$. Then $E(N)=\Sum{r=0}{k}\frac{N^r}{r!}$ so that $N\mapsto E(N)-id$ is obviously algebraic.

On the other hand, define $L:Nil(m)\rightarrow Nil(m)$ given by 
\begin{center}
	$N\mapsto \log(id+N)=\Sum{r=1}{\infty}\frac{(-1)^{r+1} N^r}{r}$.
\end{center}Since our operators are nilpotent this sum is finite and the map is algebraic, similarly to the above argument. The two functions are inverse of each other, which proves the claim.  \end{proof}

\begin{rmk}
	This shows that the $T_j$ are the natural generalization of the notion of ``eigencoordinates" we saw in the case of $\mathfrak{h}_n$. We note that in the case of $\mathfrak{h}_n$ instead of the more involved spaces $Y_k(\vec{m})$ we had $\mathbb{A}^k_\C\times \mathbb{A}_\C^{\frac{n(n-1)}{2}}$. The role of $\mathbb{A}_\C^{\frac{n(n-1)}{2}}$, the space of the eigencoordinates, is now played by $W_k(\vec{m})$.

\end{rmk}
\section{Ax-Schanuel for $GL_n$}

We continue with our study of $E:\mathfrak{gl}_n\rightarrow GL_n$ by stating the Ax-Schanuel result and proving it by reducing to \thref{propinter}.

We will denote the coordinates of the map $E$ by $E_{i,j}$. We start with stating the theorem in the functional point of view.

\begin{theorem}[Weak Ax-Schanuel for $GL_n$]\thlabel{asgln} Let $g_{i,j}\in\C[[t_1,\ldots,t_l]]$ be power series with no constant term, where $1\leq i,j\leq n$. Let $f_i$, where $1\leq i\leq n$, denote the eigenvalues of the matrix $A=(g_{i,j})$. Let us also set $N=\dim_\Q\langle f_1,\ldots,f_n\rangle_{\Q}$, then 
	\begin{center}$tr.d._\C \C(A,E(A))\geq N+\rank J((g_{i,j});\vec{t})$.\end{center}\end{theorem}
\begin{proof} Let $\mathbb{L}$ be the field of Puiseux series defined above. Assume $A=(g_{i,j})\in \mathfrak{gl}_n(\C[[t_1,\ldots, t_l]])$ has exactly $k$ distinct eigenvalues. Let us also assume that the data of the matrix $A$ is given by \begin{center}
		$(f_1,\ldots,f_k)$,$(m_1,\ldots,m_k)$, $(V_1,\ldots, V_k)$, and $(N_1,\ldots,N_k)$.
	\end{center} Let $T_j=T_j(A)$, $1\leq j\leq w_{k,\vec{m}}$ be the coordinates of the point in $W_{k}(\vec{m})(\mathbb{L})$ describing the above data of $A$.

We are therefore in a position to apply \thref{trdataexp} to get that
\begin{center}
	$tr.d._\C \C(A,E(A))=tr.d._\C \C(\{f_i,e^{f_i},T_j:1\leq i\leq k, 1\leq j\leq w_{k,\vec{m}}\})$.
\end{center}Using this together with \thref{propinter}, applied to the field $\mathbb{L}$, we get that \begin{center}
	$tr.d._\C \C(A,E(A))\geq N+rank(J(\vec{f},\vec{T};\vec{t}))$.
\end{center}
Following the remarks at the end of \ref{datasection}, the map $\Phi_{k,\vec{m}}$ is \'etale so
\begin{center}
	$\rank J((g_{i,j});\vec{t})=rank(J(\vec{f},\vec{T};\vec{t}))$,
\end{center}and the result follows.\end{proof}

\textbf{An Alternate Formulation\\}

Similar to the alternate formulation \thref{aslast} for the Ax-Schanuel we had for $U_n$  we can give an alternate form of \thref{asgln}, again we have to change the background accordingly.

 We let $V\subset \mathfrak{gl}_n$ be an open subset and $X\subset V$ an irreducible complex analytic subvariety of $V$ containing the origin such that locally at $\vec{0}$ the functions $g_{i,j}$ for $1\leq i,j\leq n$, are meromorphic functions on $X$.

Once again, for reasons of convenience, and notational coherence, we let $A=(g_{i,j})$ denote the matrix corresponding to the coordinates $g_{i,j}$.\\

\begin{theorem}[Weak Ax-Schanuel-Alternate Formulation]\thlabel{asalt} In the above context, if the eigenvalues $f_i$ of the matrix $A=(g_{i,j})$ are such that $N=\dim_\Q\langle f_1,\ldots,f_n\rangle_{\Q}$, then 
	\begin{center}$tr.d._\mathbb{C}\mathbb{C}(A, E(A))\geq N+\dim_\C X$.\end{center}
\end{theorem}\begin{proof}
 We choose  $t_1,\ldots,t_l$ that are independent holomorphic coordinates on $X$ locally at $\vec{0}$ so that $\rank(J((g_{i,j};\vec{t}))=dim_\C(X)$. This reduces the proof to \thref{asgln}
\end{proof}

This version of the Ax-Schanuel result is the one most useful when extracting geometric corollaries, as we have already seen.

\section{The Weakly Special Subvarieties}

We return once more to $\mathfrak{gl}_n$ and proceed towards defining the weakly special subvarieties. The results we had so far lead us in a natural way to consider some specific subsets of $\mathfrak{gl}_n$. 

Consider a vector space $V$ over $\C$ with $\dim_\C V=n$, some fixed $k\in\{1,\ldots,n \}$, some fixed $m_i\in \N$ for $1\leq i\leq k$ such that $m_1+\ldots+m_k=n$ and $W_k(\vec{m})$ the algebraic variety over $\C$ we defined earlier.\\

\textbf{Relations on Eigenvalues}\\

We expect that the only algebraic relations that will be allowable on the eigenvalues will be $\mathbb{Q}-$linear relations. Since we have already accounted for the number of distinct eigenvalues we also require that these relations do not force any more eigevalues to be equal.

With that in mind, we let $\Sigma$ be a finite set of $\Q-$linear polynomials of the form $F(\vec{f})=\Sum{i=1}{k}q_if_i$ on the $f_i$, and let $I$ be the ideal generated by $\Sigma$ in $\C[f_1,\ldots,f_k]$. We also assume that $\nexists i,j$ such that $f_i-f_j\in I$, i.e. in $Z(\Sigma)$ none of the previously distinct $f_i$ coincide, where $Z(\Sigma)$ is the algebraic subvariety of $\mathbb{A}_{\C}^k$ defined by $\Sigma$. Finally, we also let \begin{center}
	$\Gamma_\Sigma=\Gamma_k\cap Z(\Sigma)\subset \mathbb{A}^k$,
\end{center}where $\Gamma_k$ is the Zariski open subset of $\mathbb{A}^k_\C$ defined earlier.\\

\textbf{Other Relations}\\

For the rest of the data of the matrix we allow any algebraic relation that does not depend on the eigenvalues. So we consider $W\subset W_k(\vec{m})$ to be a subvariety of $W_k(\vec{m})$. Then if we are given a set $\Sigma$ as above and a subvariety $W\subset W_k(\vec{m})$ we let
\begin{center}$X(k,\vec{m},\Sigma,W)=\Gamma_\Sigma\times W$.\end{center}

All of the above lead us naturally to the following definition.

\begin{defn}An irreducible subvariety $U\subset \mathfrak{gl}_n$ containing the origin will be called \textbf{weakly special} if there exist a natural number $1\leq k\leq n$, a vector $\vec{m}=(m_1,\ldots,m_k)\in\N^k$ such that $\Sum{i=1}{k}m_i=n$, a set $\Sigma$ of  $\Q-$linear polynomials, and a subvariety $W\subset W_k(\vec{m})$, all defined as above, such that \begin{center}
		$U=Zcl(\Phi_{k,\vec{m}}(X(k,\vec{m},\Sigma,W))$,
	\end{center} where $Zcl(R)$ denotes the Zariski closure in $\mathfrak{gl}_n$ of a subset $R\subset \mathfrak{gl}_n$.\end{defn}

\section{Ax-Lindemann for $GL_n$ and other corollaries}

We approach this in the same way as we did for the corresponding result in \ref{saxlindup}. We start with defining components in this setting. After that we prove a two-sorted version of \thref{asalt}, similar to \thref{atypint}, and then, just as in \ref{saxlindup}, use this to infer our Ax-Lindemann result.\\

\begin{defn} Let $W\subset \mathfrak{gl}_n$ and $V\subset GL_n$ be algebraic subvarieties. Then a \textbf{component} $C$ of $W\cap E^{-1}(V)$ will be a complex-analytically irreducible component of $W\cap E^{-1}(V)$.\end{defn}

\begin{theorem}[Two-sorted Weak Ax-Schanuel for $GL_n$]\thlabel{atypintgln}  Let $U\subset \mathfrak{gl}_n$ be a weakly special subvariety that contains the origin, let $X=E(U)$, and let $V\subset U$ and $Z\subset X$ be algebraic subvarieties, such that $\vec{0}\in V$ and $I_n\in Z$. If $C$ is a component of $V\cap E^{-1}(Z)$ with $\vec{0}\in C$, then, assuming that $C$ is not contained in any proper weakly special subvariety of $U$,
\begin{center}$\dim_\C C\leq \dim_\C V+\dim_\C Z-\dim_\C X$.
\end{center}\end{theorem}

\begin{proof} Let $U=Zcl(\Phi_{k,\vec{m}}(X(k,\vec{m},\Sigma,W)))$, for brevity we let $B_k:=X(k,\vec{m},\Sigma,W)$, and $B=\Phi_{k,\vec{m}}(B_k)\cap C$. We will have $\dim_\C B=\dim_\C C$ and on $B$ the $g_{i,j}$ are well defined meromorphic functions. We let $K$ be the algebraic closure of the field $\C((g_{i,j}))$. Let $f_i$ and $T_j$ denote the coordinates of the point in $T_\Sigma(K)$ and $W\subset W_k(\vec{m})(K)$ respectively with $\Phi_{k,\vec{m}}((\vec{f},\vec{T}))=g_{i,j}$. We also let $N=\dim_\Q\langle f_1,\ldots,f_k\rangle_{\Q}$.

Then we may use \thref{asalt} to deduce that 
\begin{equation}\label{eq:last1}
tr.d._\C \C(A,E(A))\geq N+\dim_\C C.
\end{equation}On the other hand, we have the following inequalities:
\begin{equation*} \dim_\C Z\geq tr.d._\C \C(E(A)), \text{ and} \end{equation*}
\begin{equation*} \dim_\C V\geq tr.d._\C \C(A).  \end{equation*}Combining these with \thref{trdata} and \thref{trdataexp} we conclude that
\begin{equation}\label{eq:dimz2} \dim_\C Z\geq tr.d._\C \C(\{e^{f_i},T_j: 1\leq i\leq k, 1\leq j\leq w_{k,\vec{m}}\}), \text{ and} \end{equation}
\begin{equation}\label{eq:dimv2} \dim_\C V\geq tr.d._\C \C(\{f_i,T_j: 1\leq i\leq k, 1\leq j\leq w_{k,\vec{m}}\}).  \end{equation}

On the other hand, if we set $L=\C(\{T_j :1 \leq j\leq w_{k,\vec{m}} \})$, and let $M=tr.d._\C L$, we get that 
\begin{equation}\label{eq:dimx2}
\dim_\C X =N+M.
\end{equation}

By the minimality of $U$, in containing $C$, and hence $B$, we also get that
\begin{equation}\begin{split}
tr.d._\C \C(\{f_i,T_j: 1\leq i\leq k, 1\leq j\leq w_{k,\vec{m}}\})=& \\
=M+tr.d._L L(\{f_i:1\leq i\leq k\}),&
\end{split}
	\end{equation}
\begin{equation}\begin{split}tr.d._\C \C(\{e^{f_i},T_j: 1\leq i\leq k, 1\leq j\leq w_{k,\vec{m}}\})=& \\
=M+tr.d._L L(\{e^{f_i}:1\leq i\leq k\}), \text{ and    }&
\end{split}\end{equation}\begin{equation}
tr.d._\C \C(A,E(A))= M+tr.d._L L(\{f_i,e^{f_i}:1\leq i\leq k\}).
\end{equation}
Combining these with \eqref{eq:dimz2} and \eqref{eq:dimv2} we get that 
\begin{center}
	$\dim_\C V\geq M+tr.d._L L(\{f_i:1\leq i\leq k\})$, and
\end{center} \begin{center}
$\dim_\C Z\geq  M+tr.d._L L(\{e^{f_i}:1\leq i\leq k\})$.
\end{center}

The rest of the proof follows similarly to that of \thref{atypint}.
\end{proof} 

As we did in \ref{saxlindup}, we conclude with the characterization of bi-algebraic sets that contain the origin.
 
\begin{cor}[Ax-Lindemann for $GL_n$]\thlabel{axlindgln}Let $Z\subset GL_n$ be an algebraic subvariety with $I_n\in Z$. If $V\subset E^{-1}(Z)$ is a maximal irreducible subvariety that contains the origin, then $V$ is a weakly special subvariety.\end{cor}
\begin{proof}Let $U$ be the minimal weakly special subvariety that contains $V$, $X=E(U)$ and let $Z'=Z\cap E(U)$. We use \thref{atypintgln} for $C=V$ to get \begin{center}
		$\dim_\C V\leq \dim_\C{V}+\dim_\C Z'-\dim_\C X$.\end{center}This implies $\dim_\C X\leq \dim_\C Z'$, and since $Z'\subset X$ we get that $X\subset Z$ and that $V\subset U\subset E^{-1}(Z)$. Maximality of $V$ then implies that $V=U$ is weakly special.   \end{proof}
	
\begin{rmks}1. We note that the above results imply, as a direct corollary, Weak Ax-Schanuel and therefore also Ax-Lindemann results for all linear algebraic groups.\\

2. In mimicking the classical Ax-Schanuel statement, we can extract Weak Ax-Schanuel and Ax-Lindemann type statements for Cartesian powers of the exponential map of a Lie algebra from our results.

Even more generally, we can infer such results for the Cartesian products of exponentials $E_i:\mathfrak{g}_i\rightarrow G_i$ of Lie algebras $\mathfrak{g}_i$, $1\leq i\leq r$. We achieve this by noticing that the exponential of the Lie algebra $\mathfrak{g}=\mathfrak{g}_1\times\ldots\times \mathfrak{g}_r$ is the Cartesian product of the $E_i$. \end{rmks}

\section{Applications and Further Questions}\label{applications}

Ax-Schanuel and Ax-Lindemann type results, proven in other settings, play a major role in the proof of problems of unlikely intersections such as Andr\'e-Oort and Zilber-Pink type statements. In this last section we discuss possible applications and extensions of our results. Mainly we propose several questions that we believe will ultimately lead to unlikely intersections statements similar in principle to those studied for other spaces.\\

\subsection*{Towards the Full Ax-Schanuel Conjecture}

The main obstacle in obtaining the Full Ax-Schanuel Conjecture for the maps considered in here seems to stem from the fact that the preimage $E^{-1}(I_n)$ of $I_n$ is no longer a discrete subset of the corresponding Lie algebra, whether that is $\mathfrak{h}_n$ or $\mathfrak{gl}_n$. In both cases it is a countable union of connected, possibly higher dimensional subsets of the Lie algebra.

We note that our method in fact is able to classify a wider class of bi-algebraic sets. In fact it is not hard to see from our exposition that in the case of the algebra $\mathfrak{h}_n$ of upper triangular matrices we can describe all bi-algebraic sets $W$ that satisfy the following condition:\begin{center}
	($C_1$) There is no $k\in\Z\backslash\{0\}$ and no pair $1\leq i<j\leq n$ such that $W\subset Z(f_i-f_j-2k\pi i)$.
\end{center}

Similarly, for $\mathfrak{gl}_n$ following our exposition one may classify all bi-algebraic subvarities $W$ satisfying the condition:\begin{center}
	($C_2$) There is no $k\in\Z\backslash\{0\}$ such that generically on $W$ there are eigenvalues that differ by $2k\pi i$.
\end{center}

This inability to classify all bi-algebraic subsets of the Lie algebra $\mathfrak{h}_n$, respectively of $\mathfrak{gl}_n$, is the reason why we have avoided giving a definition of ``weakly special" subvarieties of $U_n$, or respectively of $GL_n$. 

We believe the first question to be considered here should be to describe the rest of the bi-algebraic subsets of the Lie algebra of upper triangular matrices $\mathfrak{h}_n$. The reason for that is two-fold. First and foremost, the case of upper triangular matrices offers examples more easily accessible from a computational standpoint. Secondly, it offers the exact same limitations as that of the algebra $\mathfrak{gl}_n$, as seen evidently from our exposition. Namely the set $E^{-1}(I_n)$ has positive dimensional connected components.

Let us call the bi-algebraic sets of $\mathfrak{h}_n$, or respectively of $\mathfrak{gl}_n$, satisfying the above condition $C_1$, or $C_2$ respectively, ``weakly special of type I". We call the rest ``weakly special of type II". Then we believe that the following holds
\begin{conj}Let $\mathfrak{g}$ be either one of the above algebras and let $V\subset\mathfrak{g}$ be a weakly special subvariety of type II. Then there exists a weakly special subvariety of type I $W\subset \mathfrak{g}$, of possibly smaller dimension than $V$, such that $E(W)=E(V)$.
\end{conj}

Assuming the validity of this conjecture, it would be natural to define the weakly special subvarieties of $GL_n$ as the images of what we call ``weakly special subvarieties of type I", i.e. the weakly special subvarieties that naturally appear in our exposition.

\subsection*{Towards a Zilber-Pink Conjecture}

The first step towards formulating unlikely intersections problems such as an Andr\'e-Oort and a Zilber-Pink type statement for linear algebraic groups would have to be the correct definition of special subvarieties of $GL_n$ starting from dimension $0$, i.e. the special points. The difficulty of this definition has already been noted in \cite{zannier}, see in particular Remark $1.1.5$ and $\S 1.3.5$.

Once again, at least in defining the special subvarities, we believe the starting point should be defining the special subvarieties of $U_n$, the group of invertible upper triangular matrices. The reasons for this are the same as those noted above.

\begin{appendices}
\renewcommand{\thecor}{\Alph{section}.\arabic{cor}}
\renewcommand{\thetheorem}{\Alph{section}.\arabic{theorem}}
	
\section{The Full Ax-Schanuel Theorem in families}
	
In this appendix, following the argument in \cite{tsimax}, we prove the ``Full Ax-Schanuel" analog of \thref{asmainred}. As a consequence we also obtain a slightly more general result that could be dubbed ``Full Ax-Schanuel in affine families". We believe the results of this section are known to experts in the field, however since we couldn't find a reference for them, and we expect that they will play a role in subsequent progress towards a ``Full Ax-Schanuel for $GL_n$", we include them in this appendix.\\

 We consider the uniformizing map $\pi_k :\mathbb{C}^n\times \mathbb{C}^k\rightarrow (\mathbb{C}^{\times})^n$, which is given by \begin{center}$(x_1,\ldots,x_n,y_1,\ldots,y_k)\mapsto (e^{x_1},\ldots,e^{x_n})$.\end{center}

We define $D_k=\Gamma(\pi_k)$, i.e. as a subset of $\mathbb{C}^n\times \mathbb{C}^k\times(\mathbb{C}^{\times})^n$ 
\begin{center}$D_k=\{(\vec{u},\vec{v}): \pi_k(\vec{u})=\vec{v}\}$.\end{center}  Furthermore, let $\pi_a$ be the projection on the first $n$ coordinates of the space $\mathbb{C}^n\times \mathbb{C}^k\times(\C^{\times})^n$, and $\pi_m$ be the projection on the last $n$ coordinates of the same space. 

Following the proof of the Full Ax-Schanuel Theorem in \cite{tsimax} we prove:
\begin{theorem}[Full Ax-Schanuel in families]\thlabel{ver2} Let $V\subset \C^n\times \C^k\times(\C^{\times})^n$ be an irreducible algebraic subvariety, and $U$ a connected complex-analytic irreducible component of $V\cap D_k$. Assuming that $\pi_m(U)$ is not contained in the coset of a proper subtorus of $(\mathbb{C}^{\times})^n$, then 
	\begin{center}$\dim_\mathbb{C} V\geq n+\dim_\mathbb{C}U$.\end{center}\end{theorem}
\begin{proof}We employ induction on $k\geq 0$. For $k=0$ this is a consequence of the Ax-Schanuel Theorem\footnote{See Theorem 1.3 in \cite{tsimax}.}.
	
	Assume that $k\geq 1$ and that the result holds for $k-1$. Then we consider the projection 
	\begin{center}
		$p_0: \C^n\times \C^k\times(\C^{\times})^n\rightarrow \C$,\end{center}
	of our space to the $(n+k)-$th coordinate, i.e.
	\begin{center}$p_0(x_1,\ldots,x_n,y_1,\ldots,y_k,z_1,\ldots,z_n)=y_k$.	\end{center}
	Let also $V_0=p_0(V)$ and, for $y\in V_0$, we consider the fibre $V_y$ of $V$ over $y$. Similarly we consider the corresponding fibre $U_y$ of $U$ over $y$. With this notation we get 
	$V=\underset{y\in V_0}{\bigcup}\{y\} \times V_y$.
	
	Since $V$ is irreducible, if $\dim(V_0)=0$ then $V_0=\{y_0\}$ will be a single point. This implies that $V=V'\times \{y_0\}$ is isomorphic to an irreducible algebraic subvariety $V'\subset\C^n\times \C^{k-1}\times(\C^{\times})^n$. In this case, $U\subset V\cap D_{k}$ is isomorphic to a connected complex-analytic irreducible component of $V'\cap D_{k-1}$ and the result follows by induction.
	
	We may therefore assume that $\dim V_0=1$. This tells us that $V_0$ contains a non-empty affine open subset of $\C$ and that for $y\in V_0$ generic we get \begin{center}
		$\dim V=\dim V_y+1$.	\end{center}
	
	The rest of the proof comprises of considering the only two possible cases for the generic behaviour of the fibres $U_y$.\\

	\textbf{First Case:} Suppose that $\pi_a(U_y)$ is generically\footnote{Generically here refers to $y\in V_0$.} not contained in the translate of a proper $\Q-$linear subspace of $\C^n$. 
	
	We have that $V_y \subset\C^n\times \C^{k-1}\times(\C^{\times})^n$ is an irreducible algebraic subvariety, and $U_y$ is a connected complex-analytic irreducible component of $V_y\cap D_{k-1}$. Therefore by the previous assumption and the inductive hypothesis we get that for $y\in V_0$ generic 	
	\begin{center}
		$\dim V_y\geq n+\dim U_y$.
	\end{center}This in turn implies that $\dim V\geq n+(1+\dim U_y)$ and, since $1+\dim U_y\geq \dim U$, the result follows.\\
	
	\textbf{Second Case:} If the assumption of the previous case does not hold, then for $y\in V_0$ chosen generically, $\pi_a(U_y)\subset \C^n$ will be contained in the translate of some proper $\Q-$linear subspace of $\C^n$. In other words $U_y\subset Z(f_y)$, where $f_y=c(y)+\Sum{i=1}{n}q_i(y)x_i\in\C[x_1,\ldots,x_n]$, is a linear polynomial with the coefficients $q_i(y)\in\Q$ and $c(y)\in \C$ depending on $y$.\footnote{Here $Z(f_y)=\{(x_1,\ldots,x_n): f_y(x_1,\ldots,x_n)=0 \}$, is just the set of solutions of $f_y=0$ in $\C^n$.}
	
	At this point we consider another projection, namely we let \begin{center}
		$p_1:\C^n\times \C^{k}\times(\C^{\times})^n\rightarrow\C^n\times \C^{k-1}\times(\C^{\times})^n$
	\end{center} be the projection given by \begin{center}$p_1(x_1,\ldots,x_n,y_1,\ldots,y_k,z_1,\ldots,z_n)=(x_1,\ldots,x_n,y_1,\ldots,y_{k-1},z_1,\ldots,z_n)$.\end{center}
	We also let $V_1=p_1(V)$, $U_1=p_1(U)$, $V_1'=Zcl(V_1)$, the Zariski closure of $V_1$, and $U_1'=\bar{U_1}$, the closure of $U_1$ with respect to the standard topology on $\C^n\times  \C^{k-1}\times (\C^\times)^n$.
	
	For these new sets we get that $V_1'$ is an irreducible subvariety of $\C^n\times C^{k-1}\times (\C^\times)^n$ and $U_1'$ is a connected irreducible complex-analytic component of $V_1'\cap D_{k-1}$. We also get that $\pi_a(U)=\pi_a(U_1)$ and hence, by the initial assumption on $U$, $\pi_a(U_1')$ is not contained in the translate of a $\Q-$linear subspace of $\C^n$. Therefore we may apply the inductive hypothesis to get 
	\begin{center}$\dim V_1'\geq n +\dim U_1'$.\end{center}
	
	From the preceding discussion we get that $\dim V\geq \dim V_1=\dim V_1'$. On the other hand, since, by assumption, $\pi_a(U)$ is not contained in the translate of a $\Q-$linear subspace of $\C^n$ then the $Z(f_y)$, and hence the $f_y$, will vary with $y$. This in turn implies\footnote{The coordinate function $y_{k}$ restricted to $U$ will depend on the rest of the coordinates of $U$.} that $\dim U=\dim U_1=\dim U_1'$.
	
	Combining all of the above we reach the conclusion.
\end{proof} 

By the same arguments as in \cite{tsimax}, the above theorem implies the following
\begin{cor} Let $D_k$ and $\pi_k$ be as above and $U\subset D_k$ be an irreducible complex analytic subspace such that $\pi_m(U)$ is not contained in a coset of a proper subtorus of $(\mathbb{C}^{\times})^n$. Then 
	\begin{center}$\dim Zcl(U)\geq n+\dim_\mathbb{C}U$.\end{center}\end{cor}

\subsection{A generalization-Affine families}

As a corollary of the above proof we are able to extract Ax-Schanuel results for a larger family of spaces. The idea is that we are able to replace $\C^k$ by a random affine variety. We approach this in a geometric setting similar to the previous subsection.

Let $W$ be an affine variety over $\C$ and let $\pi_n:\C^n\rightarrow (\C^{\times})^n$ be the map given by \begin{center}$(x_1,\ldots,x_n)\mapsto(e^{x_1},\ldots,e^{x_n})$.\end{center}
We consider the uniformizing map $\pi_n\times id_W :\mathbb{C}^n\times W\rightarrow (\mathbb{C}^{\times})^n\times W$, the product of $\pi_n$ and the identity morphism $id_W$ of $W$. Let also $p_1: (\mathbb{C}^{\times})^n\times W\rightarrow (\mathbb{C}^{\times})^n$ be the projection on $(\mathbb{C}^{\times})^n$ and let $\phi:\mathbb{C}^n\times W\rightarrow (\mathbb{C}^{\times})^n $ be its composition with $\pi_n\times id_W$.

We also define $D_k(W)=\Gamma(\phi)$, i.e. as a subset of 
$\mathbb{C}^n\times W\times(\mathbb{C}^{\times})^n$ 
\begin{center}$D_k(W)=\{(\vec{u},\vec{v}): \phi(\vec{u})=\vec{v}\}$.\end{center}

\begin{cor}[Full Ax-Schanuel in Affine families]\thlabel{corvar1} Let $V\subset \C^n\times W\times(\C^{\times})^n$ be an irreducible algebraic subvariety, and $U$ a connected complex-analytic irreducible component of $V\cap D_k(W)$. Assuming that $\pi_m(U)$ is not contained in the coset of a proper subtorus of $(\mathbb{C}^{\times})^n$, then 
	\begin{center}$\dim_\mathbb{C} V\geq n+\dim_\mathbb{C}U$.\end{center}\end{cor}
\begin{proof}By Noether's Normalization Lemma there exists a finite surjective morphism $f: W\rightarrow \mathbb{A}_\C^d$ where $d=\dim W$. The product of this morphism with the identity of $\C^n\times(\C^{\times})^n$ in turn gives a finite morphism \begin{center}
		$F:\C^n\times W\times (\C^{\times})^n\rightarrow \C^n\times \mathbb{A}_\C^d\times (\C^{\times})^n$.
	\end{center}Indeed a morphism of affine varieties is finite if and only if it is proper\footnote{See Exercises II.4.1 and II.4.6 in \cite{hhorne}.}, since $F$ is proper as the product of two such morphisms it will also be finite. The image of the irreducible subvariety $V$  under this map will be an irreducible subvariety $V'$ of $\C^n\times \mathbb{A}_\C^d\times (\C^{\times})^n$, since finite morphisms are closed.
	
	We also note that $F$ maps the set $D_k(W)$ to the set $D_k(\mathbb{A}_\C^d)$. So that the closure $U'=cl(F(U))$ of the image of $U$ with respect to the Euclidean topology will be a component of $V'\cap D_k(\mathbb{A}_\C^d)$. 
	
	Since $F$ is finite we get that $\dim_\C U=\dim_\C U'$, $\dim_\C V=\dim_\C V'$, and by the construction of $F$ it follows that 
	$\pi_m(U')$ is not contained in the coset of a proper subtorus of $(\mathbb{C}^{\times})^n$, since this is true for $\pi_m(U)$. Therefore the result follows from \thref{ver2}.\end{proof}

\end{appendices}

\bibliographystyle{alpha}

\bibliography{ms}

\end{document}